\documentclass[12pt, a4paper]{siamart220329}
\usepackage{amsmath,amsfonts,amssymb}
\pdfoutput=1
\usepackage{graphics}
\usepackage{graphicx}  
\usepackage{float}
\usepackage{lineno}
\usepackage{xcolor}
\usepackage[inline]{enumitem}
\usepackage{verbatim} 
\usepackage{hyperref}
\usepackage{mathtools}
\usepackage[left=30mm, right=30mm, top=30mm, bottom=30mm]{geometry}
\usepackage{caption}
\usepackage{pifont}
\newtagform{eqta}{(}{.a)}
\newtagform{eqtb}{(}{.b)}

\theoremstyle{plain}

\newsiamremark{remark}{Remark}
\newsiamthm{assum}{Assumption}

\newcommand{\norm}[1]{\lVert #1 \rVert}
\newcommand{\abs}[1]{\lvert #1 \rvert}
\newcommand{\seminorm}[1]{\lvert #1 \rvert}
\newcommand{\imagunit}{\mathrm{i}}

\DeclareMathOperator*{\inff}{inf\vphantom{p}}
\newcommand{\RN}[1]{%
  \textup{\uppercase\expandafter{\romannumeral#1}}%
} 

\newlist{AssumpList}{enumerate}{1}
\setlist[AssumpList,1]{
  before=\setcounter{AssumpListi}{\value{AssumpList}},
  after=\setcounter{AssumpList}{\value{AssumpListi}},
}
\newcounter{AssumpList}

\newcommand{\BowenRevise}[1]{{\color{violet}{#1}}}

\def\Om{\Omega}  
\newcommand{\q}{\quad}

\def\Om{\Omega}  
\def\vs{\vskip-0.6truecm}

\numberwithin{equation}{section}

\title{Two-level hybrid Schwarz Preconditioners for \\
The Helmholtz Equation with high wave number}

\author{Peipei Lu\thanks{Department of Mathematics Sciences, Soochow University, Suzhou, 215006, China (\email{pplu@suda.edu.cn})} \and Xuejun Xu\thanks{School of Mathematical Sciences, Tongji University, Shanghai 200092, China (\email{xuxj@tongji.edu.cn})}
\and Bowen Zheng\thanks{LSEC, ICMSEC, Academy of Mathematics and Systems Science, Chinese Academy of Sciences, Beijing 100190, China; School of Mathematical Sciences, University of Chinese Academy of Sciences, Beijing 100049, China ({\email{zhengbowen@lsec.cc.ac.cn})}} \and Jun Zou\thanks{Department of Mathematics, The Chinese University of Hong Kong, Shatin, N.T., Hong Kong SAR. The work of this author was substantially supported by Hong Kong RGC General Research
Fund (projects 14306921 and 14308322) and NSFC/Hong Kong RGC Joint Research Scheme (project N\_CUHK465/22). (\email{zou@math.cuhk.edu.hk})}}

\date{}

\begin{document}

\maketitle

\begin{abstract}
In this work, we propose and analyze two two-level hybrid Schwarz preconditioners for solving 
the Helmholtz equation with high wave number in two and three dimensions. Both preconditioners 
are defined over a set of overlapping subdomains, with each preconditioner formed by 
a global coarse solver and one local solver on each subdomain. 
The global coarse solver is based on the localized orthogonal decomposition (LOD) technique, 
which was proposed in \cite{malqvist_localization_2014, peterseim_eliminating_2017} 
originally for the discretization schemes for elliptic multiscale problems with 
heterogeneous and highly oscillating coefficients and Helmholtz problems 
with high wave number to eliminate the pollution effect. The local subproblems 
are Helmholtz problems in subdomains with homogeneous boundary conditions 
(the first preconditioner) or impedance boundary conditions (the second preconditioner).
Both preconditioners are shown to be optimal under some reasonable conditions, 
that is, a uniform upper bound of the preconditioned operator norm 
and a uniform lower bound of the field of values are established 
in terms of all the key parameters, such as the fine mesh size, 
the coarse mesh size, the subdomain size and the wave numbers. 
It is the first time to show that the LOD solver can be a very effective coarse solver when it is used appropriately 
in the Schwarz method with multiple overlapping subdomains. 
Numerical experiments are presented to confirm the optimality and efficiency of the two proposed preconditioners.

\end{abstract}
\section{Introduction}
In this work we consider the Helmholtz equation with an impedance boundary condition, 
which is the first order approximation of the Sommerfeld's condition \cite{radiationcondition}:
\begin{equation} \label{eqn:strongform}
    \begin{cases}
        -\Delta u - \kappa ^2 u = f,  &\ \mathrm{in} \ \Omega, \\
        \partial u / \partial n - \imagunit \kappa  u = g, &\ \mathrm{on} \  \Gamma,
    \end{cases}
\end{equation}
where $\Omega \subset \mathbb{R}^{d}$ ($d =2,3$) is a convex polygonal or polyhedral domain, with 
$\Gamma =\partial \Omega$ as its boundary and $n$ as the outward normal to its boundary. 
The constants $\kappa > 0$ and $\imagunit=\sqrt{-1}$ denote the wave number and 
the imaginary unit respectively. 
The Helmholtz equation is of great interest in many important applications, such as 
in the acoustic and electromagnetic scattering \cite{Colton_Inverse_1998}.
It has still remained to be a great challenge how to solve the Helmholtz equation efficiently, especially with 
high wave numbers. 

There are two main difficulties in numerical solutions of the Helmholtz equation with high wave number. 
Firstly, because of the pollution effect, increasingly refined meshes are necessary to accurately capture 
the oscillatory solutions when the standard Galerkin methods are used, leading to  
large-scale discretized systems. Secondly, due to the strong indefiniteness of the resulting discrete 
systems, standard fast solvers like domain decomposition methods fail to converge. 
%
During the past three decades, many attempts have been made to develop efficient preconditioners 
for solving the large-scale discrete systems arising from non-symmetric and non-coercive elliptic problems. 
The classical additive Schwarz domain decomposition method (DDM) with the standard finite element coarse solver 
and Laplace-like subproblems was shown \cite{cai_domain_1992} 
to have the same optimal convergence as for symmetric positive elliptic problems,  
i.e., the rate of convergence is independent of the fine and coarse mesh sizes if the coarse mesh is sufficiently fine. 
Similar results were developed also for unstructured meshes \cite{Chan_Convergence_1996}. 
Various efforts have been also made for developing DDMs for Helmholtz-type equations. 
A non-overlapping DDM was proposed in \cite{benamou_domain_1997} by employing 
the impedance condition, and an optimized Schwarz method 
was studied in \cite{gander_optimized_2002} while some 
nonoverlapping DDMs were proposed in \cite{chen_robust_2016}, with 
Robin transmission conditions over the subdomains.  All these DDMs were 
focused on the choice of effective interface conditions. Furthermore, two sweeping-type preconditioners 
were proposed for a centered difference scheme applied to the Helmholtz equation \cite{engquist_sweeping_2011_1,engquist_sweeping_2011_2}, 
based on an approximate $LDL^t$ factorization by sweeping the domain layer by layer starting from an
absorbing layer.

Recently, there are quite some studies of the so-called ``shifted-Laplace"  problem:
\begin{equation}
\label{1.2}
\begin{cases}
-\Delta{u}-(\kappa ^2+i\varepsilon)u=f    &\quad {\rm in} \;\Omega,\\
\frac{\partial u}{\partial n}-i\eta u=g              &\quad {\rm on}\;\Gamma
\end{cases}
\end{equation}
for certain parameter $ \eta=\eta(\kappa ,\varepsilon) $, which is usually chosen to be either $ sign(\varepsilon)\kappa $ or $ \sqrt{\kappa ^2+i\varepsilon} $. Let us denote the system arising from the continuous piecewise linear finite element approximations of $ (\ref{1.2}) $ by $ A_\varepsilon $ (or simply $ A $ when $ \varepsilon=0 $). The idea of ``shifted-Laplace " preconditioning is to construct a preconditioner
$ B_\varepsilon^{-1}$ for the discrete system from the problem (\ref{1.2}) with absorption, 
then it is served as a preconditioner for the pure Helmholtz equation (\ref{eqn:strongform}). Based on the equality
\begin{align*}
B_\varepsilon^{-1}A=B_\varepsilon^{-1}A_\varepsilon  - B_\varepsilon^{-1}A_\varepsilon(I-A_\varepsilon^{-1}A),
\end{align*}
we can see that $ B_\varepsilon^{-1}$ is a good preconditioner for $A$ provided
$$
(\mathbf{i}) \ A_\varepsilon^{-1} \; {\rm is \;  a \; good \;  approximation \; of} \; A;  \quad 
(\mathbf{ii}) \ B_\varepsilon^{-1} \;{\rm  is \;  a \; good \;  preconditioner \; for} \; A_\varepsilon.  
$$
It was proven that under general conditions on the domain and mesh sequence, 
$(\mathbf{i})$ holds 
when $\vert \varepsilon \vert/\kappa$ is bounded above by a sufficiently small constant \cite{gander_applying_2015}.
On the other hand, it was demonstrated in \cite{graham_domain_2017} that 
for $|\varepsilon|\approx \kappa ^2 $, classical overlapping additive Schwarz method performs 
optimally for the absorptive problem, i.e.,  $(\mathbf{ii})$ holds. Unfortunately, there is no suitable $\varepsilon $ to satisfy   $(\mathbf{i})$ and  $(\mathbf{ii})$ simultaneously for the classic two-level overlapping DDM. 
One-level additive Schwarz preconditioners were considered in \cite{graham_domain_2020}
for the problem (\ref{1.2}) by
solving independent local Helmholtz impedance subproblems on overlapping subdomains, 
and $(\mathbf{ii})$ was established rigorously, under a much lower level of absorption 
but with some additional conditions on the subdomain and overlapping sizes.

It is well known that the coarse space plays an important role  in the convergence of DDMs, especially for the 
scalability of the DDMs when the number of coarse subdomains becomes large. 
An empirical two-level DD preconditioner was introduced in \cite{conen_coarse_2014} for the Helmholtz equation, 
where the coarse space correction involves local eigenproblems of Dirichlet-to-Neumann (DtN) maps, 
and this coarse space was further exploited in \cite{chupeng_wavenumber_2023}.
In \cite{kimn_restricted_2007}, some two-level overlapping DD preconditioners based on 
the balancing DD and restricted additive Schwarz methods were investigated. 

Localized orthogonal decomposition (LOD) method was motivated by eliminating the pre-asymptotic and resonance effects in some multiscale problems \cite{malqvist_localization_2014}, and was further generalized to the Helmholtz 
equation to suppress the pollution error \cite{peterseim_eliminating_2017}.  It was proved \cite{peterseim_eliminating_2017} that the method is
stable and the pollution effects are eliminated when
$\kappa  H\lesssim 1$ and $m \gtrsim \log \kappa $, 
where $H$ and $m$ are the mesh size and oversampling index.
In this work, we propose a new class of two-level hybrid Schwarz preconditioners for Helmholtz equation 
(\ref{eqn:strongform}). The word 'hybrid' here means the coarse space component of our methods 
is multiplicative while the subdomain solvers remain additive. 
Our key idea is to use the LOD scheme developed in \cite{peterseim_eliminating_2017} 
to construct the coarse solver.
The new preconditioners will be used in the outer GMRES iteration. 
A two-level hybrid DD preconditioner was used for the time-harmonic Maxwell equation in \cite{bonazzoli_domain_2019} for 
the large absorption case. The optimality was proven when $\varepsilon\sim \kappa^2$, 
so the gap between $(\mathbf{i})$ and $(\mathbf{ii})$ still remains as in the Helmholtz case. This hybrid method is related to the deflation method and balancing Neumann-Neumman method in the PCG iteration of SPD matrices \cite{tang_comparison_2009}.

Our new preconditioners will be developed directly for Helmholtz equation (\ref{eqn:strongform}), 
but can be naturally extended for the absorption model \eqref{1.2}, for which we are able 
to significantly weaken the condition of $\varepsilon$  in $(\mathbf{ii}) $.
More specifically, we use the Petrov-Galerkin multiscale scheme 
in \cite{peterseim_eliminating_2017} to construct our coarse solver, which makes (ii) hold
even when $|\varepsilon| \lesssim \kappa$, including 
$\varepsilon=0$ (i.e., the Helmholtz equation). 
Two types of preconditioners are developed, 
under the condition that the coarse grid and subdomains are of the size $O(\kappa ^{-1})$ 
for the first preconditioner, but that the coarse grid is of the size $O(\kappa ^{-1})$ and 
the overlapping of subdomains is not too small (i.e., $\gtrsim \kappa^{-1}$) 
for the second preconditioner.
Both new preconditioners are proved to be optimal, namely, independent of the key parameters
$h$, $H$ and $\kappa $, by developing uniform upper and lower bounds on the preconditioned operator norm 
and the field of values, respectively. 

We have noticed a very recent work \cite{hu_novel_2024}, where a two-level hybrid restricted 
Schwarz preconditioner was developed for the Helmholtz problem \eqref{1.2} 
with absorption $\varepsilon$ in two dimensions.
Similarly to \cite{conen_coarse_2014,bootland_overlapping_2023,chupeng_wavenumber_2023,hu_novel_2024}, 
the coarse space is spanned by eigen-functions of some local generalized eigenvalue problems,
which are defined by weighted positive semi-definite bilinear forms on subspaces
consisting of local discrete Helmholtz-harmonic functions with impedance boundary data. 
The preconditioners were proved to be optimal, namely, independent of 
the key parameters $h$, $H$ and $\kappa $ for the Helmholtz equation ($\varepsilon=0$)
under certain assumptions, although it is unclear if they would be still $h$-independent  
for fixed $\kappa$ as it involves a factor $(\kappa h)^{-1/2}$ in the analysis.
It has remained a challenging topic in the last three decades to provide a rigorous verification of 
the optimality of a two-level Schwarz preconditioner, and  
this appears to be the first successful effort.

The design of our coarse solver is very different from \cite{conen_coarse_2014,bootland_overlapping_2023,chupeng_wavenumber_2023,hu_novel_2024}, 
where the coarse solvers are built up on the local eigenvalue problems in some special 
discrete Helmholtz harmonic spaces, which should be achieved by solving $n_{sub}$ 
local Helmholtz problems on each subdomain with impedance boundary conditions, 
with $n_{sub}$ being the number of fine nodal points on the boundary of each subdomain. 
Our coarse solver appears to be much simpler, and 
the construction of coarse basis functions is completely local and parallel. 
There is a major and important difference between our coarse solvers and those in 
\cite{conen_coarse_2014,bootland_overlapping_2023,chupeng_wavenumber_2023,hu_novel_2024}, 
i.e., our coarse solver 
is independent of the local subproblems so the subdomains can vary and be replaced freely 
with no effect on the coarse solver and the optimality of the resulting preconditioners. 
Although we require the overlapping parameter $m$ to satisfy 
$m \gtrsim \log \kappa $ in the convergence analysis of our algorithms, 
we have observed from all our numerical experiments that $m$ can be taken very small, 
mostly $m=1$ or $2$ is sufficiently enough to guarantee the stable convergence of our algorithms 
even for the wave number up to $\kappa =500$.

The organization of 
this paper is as follows: Some basic properties of finite element methods, 
domain decomposition methods and the LOD method are presented in Section \ref{sec:pre}.
In Section \ref{sec:LODCoarse}, we present the LOD coarse space and coarse residual analysis and Section \ref{sec:hybriddd} is dedicated to the analysis of 
the new hybrid DD preconditioners. In Section \ref{sec:numexp}, 
we show some numerical results to demonstrate the robust performance of the new preconditioners. 
Throughout this paper, we use the notation $A\lesssim B$ and
$A\gtrsim B$ for the inequalities $A\leq CB$ and $A\geq CB$, where
$C$ is a positive generic constant independent of all the key 
parameters in the relevant analysis, such as the mesh size $h$, the overlap $\delta$, the coarse mesh size 
$H$, the wave number $\kappa $, etc.. 
\BowenRevise{We are interested in the high wave number case, so in this paper we assume $\kappa \gtrsim 1$.}

\section{Preliminaries}\label{sec:pre}
This section states the relevant definitions and lemmas, and introduces domain decompositions 
and coarse spaces. To do so, we first present the variational form of the Helmholtz equation 
\eqref{eqn:strongform}, namely, we seek $u \in H^{1}(\Omega)$ such that 
\begin{equation}\label{eqn:variationform}
    a(u,v):= \int_{\Omega}\nabla u\nabla \bar{v} - \kappa^2 \int_{\Omega}u \bar{v} - \imagunit \kappa \int_{\Gamma} u\bar{v} = \int_{\Omega}f \bar{v} + \int_{\Gamma}g\bar{v}, \ \text{for}\ \text{all} \, v\in H^{1}(\Omega).
\end{equation}
The system (\ref{eqn:variationform}) is 
well-posedness, and with the stability estimate \cite{hetmaniuk_stability_2007,cummings_sharp_2006}
\begin{equation}\label{eq:apriori_estimate}
    \norm{u}_{1,\kappa} \leq C_{st} (\norm{f}_{L^2(\Omega)}+ \norm{g}_{L^2(\Gamma)})
\end{equation}
for some $C_{st}$ that may depend on $\kappa$ and $\Omega$. Here 
$\norm{\cdot}_{1,\kappa }$ is the Helmholtz energy-norm (cf.\,(\ref{energy norm })).

\subsection{Finite Element Method and Domain Decomposition}\label{subsec:fedd}
Let $\mathcal{T}^h$ be a family of conforming shape-regular triangulations as the mesh 
diameter $h\to 0$. Then we introduce an approximation spaces $V_h$
on $\mathcal{T}^h$, which consists of continuous and piecewise polynomials of degree $r\geq 1$. 
We assume $\mathcal{T}^h$ has the nodal set $\mathcal{N}^h=\{\mathbf{x}_{q} : q \in \mathcal{I}^h\}$, 
with $\mathcal{I}^h$ being the nodal index set, and  $V_h$ has a nodal basis $\{\phi_p : p \in \mathcal{I}^h\}$ 
with $\phi_{p}(\mathbf{x}_{q}) = \delta_{p,q}$.
For any continuous function on $\bar \Omega $,  we denote the standard nodal interpolation operator 
on $V_h$ by $\Pi_h$. 
Then the finite element approximation of the weak form \eqref{eqn:variationform} is to find $u_h \in V_h$ such that 
\begin{equation}\label{eqn:Galerkinform}
    a(u_h,v_h) = (f,v_h), \text{\ for\ all\ } v_h \in V_h.
\end{equation}
As a standard condition to ensure the well-posedness of \eqref{eqn:Galerkinform} and the convergence 
of the finite element solution, we will assume $\kappa h \ll 1$ in the entire paper. 

%
We consider an overlapping decomposition of subdomains of $\Omega$, 
$\{\Omega_{\ell}\}_{\ell=1}^{N}$, with each $\Omega_\ell$ being 
non-empty and formed by a union of elements of the mesh $\mathcal{T}^h$. \BowenRevise{We assume $\{\Omega_{\ell}\}$ are star-shaped uniformly in $\ell$ with respect to a ball \cite{graham_domain_2020}.}
We denote by $H_{\ell}$ the characteristic length of $\Omega_{\ell}$. 
To be more specific about the overlap, for each subdomain $\Omega_{\ell}$, 
we write $\overset{\text{\tiny\ensuremath\circ}}{\Omega}_{\ell}$ for the interior 
of $\Omega_{\ell}$ that is not overlapped by any other subdomains, and 
write $\Omega_{\ell,\mu}$ for the set of points located in distance no more than 
$\mu$ from $\partial \Omega_{\ell}$. Then we assume that there exist constants $\delta_{\ell}>0$ 
and $0<c<1$ such that the following holds for each subdomain $\Omega_\ell$: 
\begin{equation*}
    \Omega_{\ell,c\delta_{\ell}} \subset \Omega_{\ell}\backslash \overset{\text{\tiny\ensuremath\circ}}{\Omega}_{\ell} \subset \Omega_{\ell,\delta_{\ell}}.
\end{equation*}
We write 
$h_{\ell}:=\underset{\tau \in \bar{\Omega}_{\ell}}{\max}h_{\tau}$, $\delta:= \underset{\ell}{\min}\delta_{\ell}$ and $H_{\mathrm{sub}}:= \underset{\ell}{\max}H_{\ell}$. 
We further assume the generous overlap:
\begin{equation*}
    \delta_{\ell} \leq H_{\ell} \leq C \delta_{\ell}, \ \text{for\ all} \ \ell\,,
\end{equation*}
as well as the finite overlap, i.e., there exists a finite constant $ \Lambda>1$ independent of $N$ such that 
\begin{equation*}
    \Lambda = \max_{\ell} \{ \# \Lambda_{\ell} \},\qquad \text{where}\ \  \Lambda_{\ell}=\{\ell^{\prime}: \bar{\Omega}_{\ell} \cap \bar{\Omega}_{\ell^{\prime}} \neq \emptyset \}.
\end{equation*}
We shall frequently use the Helmholtz energy-norm 
\begin{equation}
\label{energy norm }
\norm{v}_{1,\kappa }:=(v,v)_{1,\kappa }^{1/2}, \qquad \text{where} \ \ (v,w)_{1,\kappa} = (\nabla v, \nabla w)_{L^2(\Omega)} + \kappa ^2(v,w)_{L^2(\Omega)},
\end{equation}
and the continuity of the sesquilinear form $a(\cdot,\cdot)$ under $\norm{\cdot}_{1,\kappa}$, i.e., \BowenRevise{there is a generic constant $C_{a}$ depends only on $\Omega$ such that}
\begin{equation} \label{eqn:acont}
    a(u,v) \BowenRevise{\leq C_{a}} \norm{u}_{1,\kappa}\norm{v}_{1,\kappa}.
\end{equation}
When $O$ is any subdomain of $\Omega$, we write $(\cdot,\cdot)_{1,\kappa, O}$ and $\norm{\cdot}_{1,\kappa, O}$ for the corresponding inner product and norm on $O$. It follows  immediately that 
\begin{equation}\label{eqn:finiteoverlap}
    \sum\limits_{\ell}^{} \norm{v}_{1,\kappa,\Omega_{\ell}}^{2} \leq \Lambda \norm{v}_{1,\kappa }^{2}\,, \quad \sum\limits_{\ell}^{} \norm{v}_{L^2(\Omega_{\ell})}^{2} \leq \Lambda \norm{v}_{L^2(\Omega)}^{2}, \ \text{for all} \ v \in H^{1}(\Omega).
\end{equation}

For each $\ell$, we introduce two approximation spaces on $\Omega_{\ell}$: 
 $V_{h,\ell}: = \{ v_h\vert_{\bar{\Omega}_{\ell}}: v_h \in V_h\}$ and
$\tilde{V}_{h,\ell}:=\{v_{h,\ell} \in V_{h,\ell}: v_h\vert_{\Gamma_{\ell}\backslash\Gamma}=0 \}$,
with $\Gamma_{\ell}:=\partial\Omega_{\ell}$. To connect $V_{h,\ell}$ to $V_h$, we use 
a partition of unity of $\Omega$, $\{\chi_{\ell}:\ell=1,\dotsc,N \}$, with the following properties \cite{toselli_domain_2005}:
\begin{equation}\label{eqn:pouprop}
    \text{supp}\, \chi_\ell \subset \bar\Omega_\ell ;  \quad 
    \chi_{\ell}(x)\geq 0, \,\,\sum\limits_{\ell}\chi_{\ell}(x)=1,\ x \in \Omega; \quad 
    \norm{\nabla \chi_{\ell}}_{\infty,\tau} \BowenRevise{\leq C_{p}} \delta_{\ell}^{-1},
    \ \text{for} \ \text{all} \ \tau \in \mathcal{T}^{h},
\end{equation}
\BowenRevise{where $C_{p}$ is independent of $\tau$, $h$ and $\ell$.}

We will analyse our two new preconditioners under the following basic assumptions:
\vs
\begin{equation}\label{eqn:resolcond}
    \kappa H \BowenRevise{\leq C_{\mathrm{coar}}} \,; \q h \le \delta\,,
\end{equation}
\BowenRevise{the value of $C_{\mathrm{coar}}$ may vary on different occasion,  but it at least satisfies \eqref{eqn:resolcond1}. If other restriction of $C_{\mathrm{coar}}$ is required, we will provide the explicit expression of it accordingly.}

\subsection{Preliminary estimates}\label{subsect:estimates}
We present in this subsection some important and helpful a priori estimates that will be used frequently 
in the subsequent analyses. 
\begin{lemma}\label{lem:multr}  
    If domain $D$ is of characteristic length $L$, it holds for any $v\in H^{1}(D)$ that \cite{grisvard_elliptic_1985}
    \begin{equation} \label{eqn:multr}
        \norm{v}_{L^2(\partial D)}^2 \BowenRevise{\leq C_{M}} \left( L^{-1}\norm{v}_{L^2(D)} + |v|_{H^{1}(D)} \right) \norm{v}_{L^2(D)}, 
    \end{equation}
    and it further holds for all $v\in H^1(D)$ 
    vanishing on a subset of $\partial D$ with measure $L^{d-1}$ that 
    \begin{align}
        \norm{v}_{L^2(D)} &\BowenRevise{\leq C_{F}} L\, \seminorm{v}_{H^1(D)},\label{eqn:trace} \\
        \norm{v}_{L^2(\partial D)} &\BowenRevise{\leq C_{tr}} \sqrt{L}\, \seminorm{v}_{H^1(D)},\label{eqn:tracecor}
    \end{align}
    \BowenRevise{where $C_{M}$ and $C_{F}$ depend only on $D$ and $C_{tr}=C_{M}^{1/2}(C_{F}(1+C_{F}))^{1/2}$.}
\end{lemma}
%

%
\begin{lemma}
    For all $w_h \in V_h$ \BowenRevise{and a generic constant $\tilde C_F>0$} satisfying 
    \begin{equation}\label{eqn:whknown}
        \norm{w_h}_{L^2(\Omega)} \BowenRevise{\leq \tilde C_F}   H\seminorm{w_h}_{H^1(\Omega)},
    \end{equation}
    we have 
    \begin{equation}\label{eqn:traceL2sub1}
        \sum\limits_{\ell}\norm{w_h}_{L^2(\Gamma_{\ell}\cap \Gamma)}^{2} \BowenRevise{\leq  C_{M}\tilde C_F} \Lambda H \norm{w_h}_{1,\kappa}^{2}\,,
    \end{equation}
    while for all $w_{\ell}\in H^{1}_{0,\Gamma_{\ell}\backslash\Gamma}(\Omega_{\ell})$, it holds that 
    \begin{equation}\label{eqn:traceL2sub2}
        \sum\limits_{\ell}\norm{w_{\ell}}_{L^2(\Gamma_{\ell}\cap\Gamma)}^{2} \BowenRevise{\leq C_{tr}^{2}}  H_{\mathrm{sub}}\sum\limits_{\ell}\norm{w_{\ell}}_{1,\kappa,\Omega_{\ell}}^{2}.
    \end{equation}
\end{lemma}
\begin{proof}
    It follows from the trace inequality \eqref{eqn:multr} that for any $w_h \in V_h$, 
    \begin{equation*}
        \norm{w_h}_{L^2(\Gamma)}^{2} \BowenRevise{\leq C_{M} } \norm{w_h}_{L^2(\Omega)}\norm{w_h}_{H^1(\Omega)},
    \end{equation*}
then \eqref{eqn:traceL2sub1} is a direct consequence of \eqref{eqn:finiteoverlap} and \eqref{eqn:whknown}, 
    \begin{equation*}
        \sum\limits_{\ell}\norm{w_h}_{L^2(\Gamma_{\ell}\cap\Gamma)}^{2} \leq \Lambda \norm{w_h}_{L^2(\Gamma)}^{2} \BowenRevise{\leq C_{M}} \Lambda \norm{w_h}_{L^2(\Omega)}\norm{w_h}_{H^1(\Omega)} \BowenRevise{\leq C_{M} \tilde C_F} \Lambda H \norm{w_h}_{1,\kappa}^{2}\,.
    \end{equation*}

    The estimate \eqref{eqn:traceL2sub2} follows readily from \eqref{eqn:tracecor}:
    \begin{equation*}
        \sum\limits_{\ell}\norm{w_{\ell}}_{L^2(\Gamma_{\ell}\cap\Gamma)}^{2} \BowenRevise{\leq C_{tr}^{2}} \sum\limits_{\ell}H_{\ell}\seminorm{w_{\ell}}_{H^{1}(\Omega_{\ell})}^{2} \BowenRevise{\leq C_{tr}^{2}} \sum\limits_{\ell} H_{\mathrm{sub}}
        \norm{w_{\ell}}_{1,\kappa,\Omega_{\ell}}^{2}. 
    \end{equation*}
\end{proof}

We shall also need the results from the following two lemmas regarding the properties of the standard 
finite element nodal interpolation and $L^2$ projection.
\begin{lemma}[\cite{graham_domain_2020}, Lemma 3.3]\label{lem:nodint}
    For every $\Omega_{l}$ and $v_h \in V_{h,\ell}$ \BowenRevise{and let $\chi_{\ell}$ be continuous and  piecewise linear on $\mathcal{T}^{h}$, there holds for some constant $C_{\Pi}$ independent of $\kappa$, $h$ and $\ell$ that }
    \begin{equation} \label{eqn:nodint}
        \left\|\left(\mathrm{I}-\Pi_h\right)\left(\chi_{\ell} v_h\right)\right\|_{1, \kappa , \Omega_{\ell}} \BowenRevise{\leq C_{\Pi}} \left(1+\kappa  h_{\ell}\right)\left(\frac{h_{\ell}}{\delta_{\ell}}\right)\left\|v_h\right\|_{H^1\left(\Omega_{\ell}\right)}.
    \end{equation}
\end{lemma}

\begin{lemma}
    The standard $L^2$-projection $\pi_{H}: {V_h} \to V_H$ is stable, i.e., \BowenRevise{there exists a constant $C_{\pi}$ depending on $\mathcal{T}^{h}$ and $\mathcal{T}^{H}$,}
    \begin{equation}\label{eqn:L2stable}
        \norm{\pi_{H} v}_{H^1(\Omega)} \BowenRevise{\leq C_{\pi}} \norm{v}_{H^1(\Omega)}, \quad \text{for all } \, v\in H^1(\Omega)\,.
    \end{equation}
\end{lemma}

We end this subsection with the introduction of two crucial sesquilinear forms that are needed later  
for defining two new preconditioners, along with their properties:
\begin{align}
    a_{\ell}(v,w) &= \int_{\Omega_{\ell}}\nabla v \cdot \nabla \bar{w} - \kappa ^2\int_{\Omega_{\ell}} v \bar{w}
    - \imagunit \kappa  \int_{\Gamma_{\ell}\cap \Gamma} v \bar{w}, \ \text{for} \ v,w \in  H^{1}(\Omega_{\ell}), 
    \label{eqn:locprobs} \\
    c_{\ell}(v,w) &= \int_{\Omega_{\ell}}\nabla v \cdot \nabla \bar{w} + \kappa ^2\int_{\Omega_{\ell}}v \bar{w} - \imagunit \kappa  \int_{\Gamma_{\ell}}v \bar{w}, \ \quad \text{for} \ v,w \in H^{1}(\Omega_{\ell}).\label{eq:locprobs_cl}
\end{align}

\begin{lemma}\label{lem:ddcontcoer}
The sesquilinear form $a_{\ell}$ is continuous:
\begin{equation}\label{eqn:contdda}
        \abs{a_{\ell}(u,v)} \BowenRevise{\leq C_{a}^{\prime}}   \norm{u}_{1,\kappa,\Omega_{\ell}} \norm{v}_{1,\kappa,\Omega_{\ell}}, \   \emph{for all }  u,v\in H^{1}_{0,\Gamma_{\ell}\backslash\Gamma}(\Omega_{\ell}),
\end{equation}
\BowenRevise{where $C_{a}^{\prime}=1+C_{M}(C_{F}+1)$.}
 Furthermore,  if $\kappa H_{\ell} \BowenRevise{\leq \frac{1}{\sqrt{2}}C_{F}^{-1}}$, it holds \BowenRevise{for $C_{a}^{\prime\prime}=1+C_{M}^{1/2} C_{tr}$} that 
    \begin{align}
        \abs{a_{\ell}(u,v)} & \BowenRevise{\leq C_{a}^{\prime\prime}}   \norm{u}_{1,\kappa,\Omega_{\ell}} \norm{v}_{1,\kappa,\Omega_{\ell}}, \ \emph{for all} \ u \in H^{1}(\Omega_{\ell}) \ ,v\in H^{1}_{0,\Gamma_{\ell}\backslash\Gamma}(\Omega_{\ell}),\label{eqn:contaesc}\\
        \Re a_{\ell}(u,u) & \BowenRevise{\geq \frac{1}{3}} \norm{u}_{1,\kappa,\Omega_{\ell}}^2,  
        \ \ \  \qquad \quad \emph{for all}\  u \in H^{1}_{0,\Gamma_{\ell}\backslash \Gamma}(\Omega_{\ell}) .\label{eqn:coeral}
    \end{align}
On the other hand,  the sesquilinear form $c_{\ell}$ is coercive, i.e., there holds for all $u,v \in H^{1}(\Omega_{\ell})$ that 
   \begin{equation}\label{eqn:coercl}
        \Re c_{\ell}(u,u) = \norm{u}_{1,\kappa,\Omega_{\ell}}^2.
    \end{equation}
 Furthermore if $\kappa \delta_{\ell} \BowenRevise{\geq} 1$, it holds \BowenRevise{for $C_{c}=1+C_{M}$} that 
    \begin{equation}\label{eqn:contddc}
            \abs{c_{\ell}(u,v)} \BowenRevise{\leq C_{c}}   \norm{u}_{1,\kappa,\Omega_{\ell}} \norm{v}_{1,\kappa,\Omega_{\ell}},\ \emph{for all } u,v \in H^{1}(\Omega_{\ell}).
    \end{equation}
    \end{lemma}

\begin{proof}
    \eqref{eqn:contdda} is a direct consequence of the Cauchy-Schwarz's inequality, \BowenRevise{\eqref{eqn:multr} and \eqref{eqn:trace}}. 
    
    For \eqref{eqn:contaesc}, we only need to show for $u\in H^{1}(\Omega_{\ell})$ and $v\in H^{1}_{0,\Gamma_{\ell}\backslash\Gamma}(\Omega_{\ell})$,
    \begin{equation}
        \kappa\abs{(u,v)_{L^2(\Gamma_{\ell}\cap \Gamma)}} \BowenRevise{\leq C_{M}^{1/2} C_{tr}} \norm{u}_{1,\kappa,\Omega_{\ell}}\norm{v}_{1,\kappa,\Omega_{\ell}}.
    \end{equation}
    It follows by applying \eqref{eqn:multr} for $u$ and \eqref{eqn:tracecor} for $v$, 
    and then using $\kappa H_{\ell} \BowenRevise{\leq} 1$ that 
    \begin{equation*}
        \begin{aligned}
        & \quad \ \kappa\norm{u}_{L^2(\Gamma_{\ell}\cap\Gamma)}\norm{v}_{L^2(\Gamma_{\ell}\cap\Gamma)} \\
        & \BowenRevise{\leq C_{M}^{1/2} C_{tr}} \kappa^{1/2}\norm{u}_{L^2(\Omega_{\ell})}^{1/2}\kappa^{1/2}(H_{\ell}^{-1}\norm{u}_{L^2(\Omega_{\ell})}+\seminorm{u}_{H^1(\Omega_{\ell})})^{1/2}H_{\ell}^{1/2}\norm{v}_{H^1(\Omega_{\ell})}\\
        & \BowenRevise{ \leq C_{M}^{1/2} C_{tr}} 
        \norm{u}_{1,\kappa,\Omega_{\ell}}^{1/2}(\kappa\norm{u}_{L^2(\Omega_{\ell})}+\kappa H_{\ell}\seminorm{u}_{H^1(\Omega_{\ell})})^{1/2}\norm{v}_{1,\kappa,\Omega_{\ell}} \BowenRevise{\leq C_{M}^{1/2} C_{tr}} \norm{u}_{1,\kappa,\Omega_{\ell}}\norm{v}_{1,\kappa,\Omega_{\ell}}.
        \end{aligned}
    \end{equation*}
    On the other hand, we derive from \eqref{eqn:trace} that 
    \begin{equation*}\begin{aligned} 
           & \Re a_{\ell}(u,u) = \seminorm{u}_{H^1(\Omega_{\ell})}^{2} - \kappa ^2 \norm{u}_{L^2(\Omega_{\ell})}^{2} \BowenRevise{\geq 1-C_{F}^2(\kappa H_{\ell})^{2}} \seminorm{u}_{H^1(\Omega_{\ell})}^{2},\\
            & \norm{u}_{1,\kappa,\Omega_{\ell}}^{2} = \seminorm{u}_{H^1(\Omega_{\ell})}^{2} + \kappa ^2 \norm{u}_{L^2(\Omega_{\ell})}^{2} \BowenRevise{\leq 1 + C_{F}^2(\kappa H_{\ell})^{2}}\seminorm{u}_{H^1(\Omega_{\ell})}^{2}.
    \end{aligned} 
    \end{equation*}
    Then \eqref{eqn:coeral} follows by noting that \BowenRevise{$\kappa H_{\ell}C_{F}\leq \frac{1}{\sqrt{2}}$}. 
    \eqref{eqn:coercl} comes directly from the definition \eqref{eq:locprobs_cl}. 
    
    By the definition of $\norm{\cdot}_{1,\kappa,\Omega_{\ell}}$, 
    to see \eqref{eqn:contddc} we only need to show for $u,v\in H^{1}(\Omega_{\ell})$ that 
    \begin{equation}\label{eq:gamma}
        \kappa\abs{(u,v)_{L^2(\Gamma_{\ell}\cap\Gamma)}}\BowenRevise{\leq C_{M}} \norm{u}_{1,\kappa,\Omega_{\ell}}\norm{v}_{1,\kappa,\Omega_{\ell}}.
    \end{equation}
    It follows from \eqref{eqn:multr} that 
    \begin{equation*}
        \begin{aligned}
            &\kappa\norm{u}_{L^2(\Gamma_{\ell}\cap\Gamma)}\norm{v}_{L^2(\Gamma_{\ell}\cap\Gamma)} 
        \\ \BowenRevise{\leq C_{M}} & \kappa^{1/2}\norm{u}_{L^2(\Omega_{\ell})}^{1/2}(H_{\ell}^{-1}\norm{u}_{L^2(\Omega_{\ell})}+\seminorm{u}_{H^1(\Omega_{\ell})})^{1/2}
        \kappa^{1/2}\norm{v}_{L^2(\Omega_{\ell})}^{1/2}(H_{\ell}^{-1}\norm{v}_{L^2(\Omega_{\ell})}+\seminorm{v}_{H^1(\Omega_{\ell})})^{1/2}\,.
        \end{aligned}
    \end{equation*}
    Now \eqref{eq:gamma} follows from the facts that $\kappa H_{\ell}\geq \kappa \delta_{\ell} \BowenRevise{\geq} 1$ and 
    the following estimate 
    \begin{align*}
    (H_{\ell}^{-1}\norm{u}_{L^2(\Omega_{\ell})}+\seminorm{u}_{H^1(\Omega_{\ell})})^{1/2} \BowenRevise{\leq} \norm{u}_{1,\kappa,\Omega_{\ell}}^{1/2}.
    \end{align*} 
\end{proof}


\subsection{Localized Orthogonal Decomposition (LOD) method}\label{subsec:premlod}
In this subsection, we first introduce the ideal and practical multiscale approximation spaces developed in \cite{peterseim_eliminating_2017},  which are the basis later for the construction of the coarse solvers 
in our new preconditioners. 

Let $\mathcal{T}^H$ be a coarse mesh of parameter $H$, which is quasi-uniform. 
Let $V_H \subset V_h$ be the space of piecewise linear functions defined on $\mathcal{T}^H$ 
with the nodal basis $\{ \Phi_p, p \in \mathcal{I}^{H} \}$ such that 
$\Phi_p(\mathbf{x}_{q})=\delta_{p,q}$, where $\mathcal{I}^{H}$ is the index set of nodal points of $\mathcal{T}^H$.
For $T\in \mathcal{T}^H$, we set $\omega(T)=\cup \{K \in \mathcal{T}^H :  K \cap \bar{T} \neq \emptyset \}$ 
for the set of all neighboring elements of $T$. For a positive integer $m$, we define 
$\omega^{m}(T)$ to be the $m$-th layer of neighbors of elements $T$, i.e., 
$\omega^{m}(T) = \omega(\omega^{m-1}(T))$. 
Let $I_H: V_h \to V_H$ be a quasi-interpolation operator, satisfying 
\begin{equation}\label{eqn:IH}
    H^{-1}\left\|v-I_H v\right\|_{L^2(T)}+\left\|\nabla I_H v\right\|_{L^2(T)} \BowenRevise{\leq C_{{I}}}
    \|\nabla v\|_{L^2(\omega(T))}, \ \text{for all} \ v \in V_h,
\end{equation}
\BowenRevise{where $C_I$ is independent of $H$, but depends on $\mathcal{T}^{H}$.}
There are different choices of $I_H$ \cite{carstensen_quasi-interpolation_1999}. For definiteness, we shall always take 
the following one: 
\begin{equation}\label{eqn:defIH}
    I_{H}v = \sum\limits_{p \in \mathcal{N}^{H}}\frac{(v,\Phi_{p})_{L^2(\Omega)}}{(1,\Phi_{p})_{L^2(\Omega)}}\Phi_{p}.
\end{equation}

The main idea of the LOD method developed 
in \cite{peterseim_eliminating_2017} is built upon suitable corrections of coarse functions by
finer ones. The kernel space gives an appropriate fine space for such corrections:
$$W_h:=\{w\in V_h: I_H w=0\}.$$
We then define a correction operator 
$\mathcal{C_{\infty}}:V_h \to W_h$. For any
$v\in V_h$,  $\mathcal{C_{\infty}}v\in W_h$ solves
\begin{equation} \label{eq:c_infinity}
    a(\mathcal{C_{\infty}}v,w)=a(v,w), \ \text{for\ all} \ w \in W_h.
\end{equation}
We also need the adjoint corrector $\mathcal{C}_{\infty}^{\star}:V_h \to W_h$ to deal with the lack of hermitivity, which solves
the adjoint variational problem. For any
$v\in V_h$,  $\mathcal{C}^{\star}_{\infty}v\in W_h$ solves
\begin{equation*}
    a(w,\mathcal{C}_{\infty}^{\star}v)=a(w,v), \ \forall w\in W_h.
\end{equation*}

Next, we define a multiscale space based on the newly defined correction operator:
\begin{equation*}
    V_{H,\infty}:=(1-\mathcal{C}_{\infty})V_H,
\end{equation*}
which is the prototypical trial space in  \cite{peterseim_eliminating_2017}. The corresponding test space is defined as 
\begin{equation*}
    V_{H,\infty}^{\star}:=(1-\mathcal{C}_{\infty}^{\star})V_H.
\end{equation*}
Note that these two spaces are actually coarse-scale spaces (with respect to the coarse meshsize $H$) in the sense that
they can be generated coarse-elementwise and by the coarse space $V_H$ with invariant dimension. 
In fact, for each $v\in V_H$,  its correction $\mathcal{C}_{\infty}v$ can be composed elementwise.
$$\mathcal{C}_{\infty}v = \sum_{T \in \mathcal T^H }\mathcal{C}_{T,\infty}(v\vert_{T}),$$
where the element contribution $\mathcal{C}_{T,\infty}(v\vert_{T})\in W_h$ solves
\begin{gather*} \label{eqn:correctbase}
    a(\mathcal{C}_{T,\infty}(v\vert_{T}), w) = a_{T}(v,w)
    \coloneq(\nabla v, \nabla w)_{L^2(T)} - \kappa ^2 (v,w)_{L^2(T)} - \imagunit \kappa  (v,w)_{L^2(\partial T\cap \Gamma)},\ 
    \text{for all}\,  w \in W_h\,.
\end{gather*}

Similarly, the dual corrections $\mathcal{C}_{\infty}^*v$ can also be achieved coarse element-wise. 

It is important to notice that the global correctors $\mathcal{C}_{\infty}$ and $\mathcal{C}_{\infty}^*$
decay exponentially \cite{peterseim_eliminating_2017}. This motivated two much more practical and localized correctors 
$\mathcal{C}_{m}$ and $\mathcal{C}_{m}^*: V_H \to W_{h}$ \cite{peterseim_eliminating_2017}. 
To specify their accurate definitions, for each $T\in \mathcal{T}^H$ we introduce a localized subspace:
$$
W_{h,T,m} = \{w_h \in W_h : \mathrm{supp}\ w_h \subset \omega^{m}(T) \}\,.
$$
Then we can define 
$\mathcal{C}_{m} = \sum_{T}\mathcal{C}_{T,m}$ and 
$\mathcal{C}_{m}^{\star} = \sum_{T}\mathcal{C}_{T,m}^{\star}$, 
where $\mathcal{C}_{T,m}: V_H\to W_{h,T,m}$ and $\mathcal{C}_{T,m}^{\star}: V_H\to W_{h,T,m}$ 
are the element correctors and 
can be achieved completely locally and indepedently, that is, 
for each $v\in V_H$, $\mathcal{C}_{T,m}(v\vert_{T})\in W_{h,T,m}$ and $\mathcal{C}^{\star}_{T,m}(v\vert_{T})\in 
W_{h,T,m}$, and solve the following equations respectively, 
\begin{align*}
    a(\mathcal{C}_{T,m}(v\vert_{T}), w) = a_{T}(v,w),\ \text{for all } w \in W_{h,T,m},\\
    a(w, \mathcal{C}^{\star}_{T,m}(v\vert_{T})) = a_{T}(w, v),\ \text{for all } w \in W_{h,T,m}\,.
\end{align*}
With the localized correctors $\mathcal{C}_{m}$ and $\mathcal{C}_{m}^*$, 
we now introduce two localized trial and test spaces: 
\begin{equation}\label{eq:VHm*}
    V_{H,m}:=(1-\mathcal{C}_{m})V_H \quad \emph{and} \quad V_{H,m}^{\star}:=(1-\mathcal{C}_{m}^{\star})V_H.
\end{equation}

From the aforementioned construction of $\mathcal{C}_{\infty}$ and  $\mathcal{C}_{\infty}^{\star}$ and a similar argument 
as the one for \cite[Lemmas 4.2 and 4.3]{peterseim_eliminating_2017}, we can derive 
the orthogonalities as well as the coercivity and stability that ensure
the well-posedness of these correction operators. 
\begin{lemma} 
    For any $v_{H,\infty}\in V_{H,\infty}$, $w_h \in W_{h}$ and 
    $z_{H,\infty}^{\star} \in V_{H,\infty}^{\star}$, it holds that 
    \begin{equation} \label{eqn:1}
            a (v_{H,\infty},w_h) = 0, \q 
            a (w_h, z_{H,\infty}^{\star}) = 0.
    \end{equation}
\end{lemma}
\begin{lemma}\label{lem:coer} 
    \BowenRevise{Let $n_o=\max _{T \in \mathcal{T}_H} \#\left\{K \in \mathcal{T}_H \mid K \subset \omega(T) \right\}$,}
    for every $w_h \in W_h$, it holds under the condition 
    \BowenRevise{
    \begin{equation}\label{eqn:resolcond1}
        \kappa H \leq \frac{1}{\sqrt{2 n_o}C_I}
    \end{equation}
    that }
    \begin{equation}\label{eqn:coerWh}
        \Re a(w_h,w_h) \BowenRevise{\geq \frac{1}{3}} \norm{w_h}_{1,\kappa }^2, 
    \end{equation}
    while the correction operators $\mathcal{C}_{\infty}$ and $\mathcal{C}_{\infty}^{\star}$ are stable:
    \begin{equation}\label{eqn:StabWh}
        \norm{\mathcal{C}_{\infty} v_h}_{1,\kappa} = \norm{\mathcal{C}_{\infty}^{\star}v_h}_{1,\kappa} \BowenRevise{\leq C_{\mathcal{C}}} \norm{v_h}_{1,\kappa},  
        \,\, {\rm for \, \,all\, } \, v_h\in V_h\, \BowenRevise{,\ C_{\mathcal{C}}=3C_{a}}.
    \end{equation}
\end{lemma}

\begin{lemma}
\label{lem:LA2}
      For $v_{H,\infty} \in V_{H,\infty}$ and $ v_{H,\infty}^{\star} \in V_{H,\infty}^{\star}$ there holds
    \begin{equation}\label{eqn:trickinfty}
        v_{H,\infty} = (I-\mathcal{C}_{\infty})\pi_{H}v_{H,\infty}, \quad v_{H,\infty}^{\star} = (I-\mathcal{C}_{\infty}^{\star})\pi_{H}v_{H,\infty}^{\star}.
    \end{equation} 
    For $v_{H,m} \in V_{H,m}$ and $v_{H,m}^{\star} \in V_{H,m}^{\star}$, the following identities are valid:    
    \begin{equation}\label{eqn:trick}
        v_{H,m} =(I-\mathcal{C}_{m})\pi_{H}v_{H,m}\,,  \quad v_{H,m}^{\star} = (I-\mathcal{C}_{m}^{\star})\pi_{H}v_{H,m}^{\star}. 
    \end{equation}
\end{lemma}
\begin{proof}
    It suffices to show the first identity since the others can be deduced  similarly. 
Let $v_{H,\infty} = (I-\mathcal{C}_{\infty})v_{H}$, where  $v_{H}\in V_{H}$. 
    It follows from the definitions of ${I}_{H}$ and $\pi_{H}$ that 
    $\mathrm{ker}\pi_{H} = \mathrm{ker}I_{H}$, therefore we know 
    $\pi_{H}w_h=0\ $ for any $ w_h\in W_h$. From this we can easily see that 
    \begin{equation*}
        \pi_{H}v_{H,\infty} = \pi_{H}v_{H} - \pi_{H}\mathcal{C}_{\infty}v_{H} = v_{H},
    \end{equation*}
    which further implies 
    \begin{align*}
        v_{H,\infty} = v_{H} - \mathcal{C}_{\infty}v_{H} = \pi_{H}v_{H,\infty} - \mathcal{C}_{\infty}\pi_{H}v_{H,\infty} = (I-\mathcal{C}_{\infty})\pi_{H}v_{H,\infty}.
    \end{align*}
\end{proof}

It is known \cite{ihlenburg_finite_1995} that $h$ should be sufficiently small to guarantee the well-posedness of \eqref{eqn:Galerkinform} on $V_h$. For a given domain $\Omega$, there exists a threshold function $\bar{h}(\kappa,r)$ such that \eqref{eqn:Galerkinform} has a unique solution for $h<\bar{h}(\kappa,r)$. Furthermore, in order for 
the localized pair of spaces $(V_{H,m},V_{H,m}^{\star})$ to approach ($V_{H,\infty},V_{H,\infty}^{\star})$, 
the oversampling parameter $m$ should be large enough. 
We make the following assumptions on $h$ and $m$ from now on for the subsequent analysis. 
\begin{assum} \label{asm:2}
\begin{enumerate*}[label=(\emph{\arabic*}), ref=(\emph{\arabic*})]
    \item $h \le \bar{h}(\kappa ,r)$;\quad \label{itm:polutefree} 
   \item \BowenRevise{ $m \geq \abs{\log \beta}^{-1} \log \big(C_{\mathrm{os}}(1+\kappa C_{st}^{\prime})\big)$\,,} \label{itm:oversample}
\end{enumerate*}\\
\BowenRevise{the value of $C_{\mathrm{os}}$ may vary but it at least satisfies \eqref{eqn:oversmplwellpose} and condition in Lemma \ref{lem:expodecay}. If other restriction of $C_{\mathrm{os}}$ is required, we will provide the explicit expression of it accordingly.}
\end{assum}
\begin{remark} 
\label{rem:polluteassum}
    Letting \BowenRevise{$C_{st}^{\prime}$} be the constant appearing in the a prior estimate 
    of \eqref{eqn:Galerkinform}, i.e., 
    \begin{equation}\label{eq:uhbound}
        \norm{u_h}_{1,\kappa} \BowenRevise{\leq C_{st}^{\prime} }(\norm{f}_{L^2(\Omega)}+ \norm{g}_{L^2(\Gamma)}), 
    \end{equation}
     then Assumption \ref{asm:2}\ref{itm:polutefree} is to ensure 
     the inf-sup condition\,\cite[Theorem 4.2]{melenk_convergence_2010}:
    \begin{equation}\label{eqn:inf_sup_Vh}
        \underset{u_h\in V_h}{\inff} \underset{v_h\in V_h}{\sup} \frac{\vert a(u_h,v_h)\vert}{\norm{u_h}_{1,k}\norm{v_h}_{1,k}} 
        \BowenRevise{\geq \frac{1}{1+ \kappa C_{st}^{\prime} } } \,.
    \end{equation}
We note that the constant \BowenRevise{$C_{st}^{\prime}$} in \eqref{eq:uhbound} may depend on $\kappa$ and 
would naturally appear also in the lower bound of the inf-sup condition 
in \eqref{eqn:inf_sup_Vh}. 
But for the sake of clarity, we shall carry out all our subsequent analyses under the condition that 
\BowenRevise{$C_{st}$ is $\kappa$-independent and there is a generic constant $C_{F\!E\!M}$ such that when $h \leq \bar{h}(\kappa,r)$ then $C_{st}^{\prime} = C_{F\!E\!M} C_{st}$.} 
Such an assumption was also made in \cite{gander_applying_2015, graham_domain_2020, wu_pre-asymptotic_2014}. 
Under this assumption, we know 
$\bar{h}(\kappa,r)\approx \kappa^{-(2r+1)/(2r)}$\cite{wu_pre-asymptotic_2014, zhu_preasymptotic_2013}. 
Indeed,  the conclusion that \BowenRevise{$C_{st}$ is $\kappa$-independent} can be made rigorously for some special domains, 
e.g., when $\Om$ is convex (cf.\,\cite{esterhazy_stability_2012}), \BowenRevise{or when $\Om$ is a smooth domain or star-shaped Lipschitz domain (cf.\,\cite{BSW15}).}
\end{remark}

Based on the exponential decay property of $\mathcal{C}_{\infty}$, the error estimates of the localized approximations of the corrector $\mathcal{C}_{\infty}$ and its adjoint $\mathcal{C}^{\star}_{\infty}$ can be easily derived
in energy-norm.

\begin{lemma}\label{lem:expodecay}
Under the condition \eqref{eqn:resolcond} and Assumption\,\ref{asm:2}\ref{itm:polutefree}, 
for any $m\in \mathbb{N}$, there exist two constants $m_0>0$ and $0<\beta<1$, 
both independent on $h$, $H$ and $\kappa $, such that the following 
estimates hold for all $v \in V_H$ and $m\geq m_0$:  
\begin{align} 
        \norm{\mathcal{C}_{\infty}v-\mathcal{C}_{m}v}_{1,\kappa ,\Omega}
        \BowenRevise{\leq C_{exp}^{\prime}} \beta^{m} \norm{\nabla v}_{L^2(\Omega)},\label{eqn:expodecay}\\
        \norm{\mathcal{C}^{\star}_{\infty}v-\mathcal{C}^{\star}_{m}v}_{1,\kappa ,\Omega}
        \BowenRevise{\leq C_{exp}^{\prime}} \beta^{m} \norm{\nabla v}_{L^2(\Omega)},\label{eqn:expodecayadj}
\end{align}
\BowenRevise{where $C_{exp}^{\prime}$ depends on $\Omega$, $C_{a}$,  $\mathcal{T}_{H}$, $I_{H}$. }
\end{lemma}

\begin{proof}
By similarity, we show only the first estimate. 
Following the proof of  \cite[theorem 5.2]{peterseim_eliminating_2017} with a minor modification 
to replace the cut-off function by its nodal interpolation, we deduce \BowenRevise{for $C_{exp}$ that depends on $\Omega$, $\mathcal{T}_{H}$ and $I_{H}$ that } 
\begin{equation}
\label{Exp_Decay_H1}
    \seminorm{\mathcal{C}_{\infty}v-\mathcal{C}_{m}v}_{H^{1}(\Omega)} \BowenRevise{\leq C_{exp} } \beta^{m}\norm{v}_{H^1(\Omega)}.
\end{equation}
  Noting the fact that $(\mathcal{C}_{\infty}v-\mathcal{C}_{m}v)\in W_h$, we further derive by using  \eqref{eqn:IH} that 
 \begin{equation*}
    \kappa ^2 \norm{\mathcal{C}_{\infty}v-\mathcal{C}_m v}_{L^2(\Omega)}^2 =
     \kappa ^2 \norm{(\mathcal{C}_{\infty}v-\mathcal{C}_m v)-I_H(\mathcal{C}_{\infty}v-\mathcal{C}_{m}v)}_{L^2(\Omega)}^2 \BowenRevise{\leq n_oC_I^{2}}
    (\kappa H)^2 \seminorm{\mathcal{C}_{\infty}v-\mathcal{C}_m v}_{H^1(\Omega)}^{2},
\end{equation*}
which, combining  \eqref{eqn:resolcond} and \eqref{Exp_Decay_H1}, yields the first estimate 
\eqref{eqn:expodecay}.
\end{proof}

\section{LOD Coarse Spaces}\label{sec:LODCoarse}
In this section we first introduce two special coarse solvers $Q_{0,m}$ and $Q_{0,m}^*$ 
that will be used in two precondirioners we propose in Section\,\ref{sec:hybriddd}, and then 
develop some nice properties of the corresponding coarse solutions and their approximations.

We first define these two coarse operators. 
The first one is a localized Petrov-Galerkin operator 
$Q_{0,m}: V_h\to V_{H,m}$. 
More specifically, for every $v\in V_h$, we seek $Q_{0,m}v\in V_{H,m}$ by solving 
the localized coarse-scale problem:
\begin{equation}\label{eqn:localizedprob}
    a(Q_{0,m}v,w_h) = a(v,w_h),\ \text{for all } w_h \in V_{H,m}^{\star}. 
\end{equation}
Similarly, we define an adjoint operator $Q_{0,m}^{\star}: V_h\to V_{H,m}^{\star}$ 
such that for every $v\in V_h$, we seek $Q_{0,m}^*v\in V_{H,m}^*$ by solving 
the localized coarse-scale problem: 
\begin{equation}\label{eqn:localizedprobadj}
    a(w_h, Q_{0,m}^{\star}v) = a(w_h,v),\ \text{for all } w_h \in V_{H,m}. 
\end{equation}

\subsection{Well-posedness of coarse solvers}\label{subsec:wellposed}
We now establish the well-posedness of two localized discrete coarse operators 
$Q_{0,m}$ and $Q_{0,m}^{\star}$. 
The well-posedness at the continuous case, i.e., the correction defined 
from $H^1(\Omega)$ to $\mathrm{ker}I_{H}:=W$, was developed in 
\cite[Theorem 5.4]{peterseim_eliminating_2017}, but more delicate analysis is needed for the discrete level here. 
\begin{lemma}
\label{lem:wellposedloc}
    Under Assumptions \ref{asm:2} and \eqref{eqn:resolcond}, \BowenRevise{especially if
    \begin{equation} \label{eqn:oversmplwellpose}
        m \geq \abs{\log \beta}^{-1} \max \{ \log 8(1+\kappa C_{st}^{\prime})(1+C_{exp}^{\prime}C_{\pi})C_{exp}^{\prime\prime} , \log 8 C_{exp}^{\prime} C_{\pi}    \},
    \end{equation}
    where $C_{exp}^{\prime\prime} = C_{a}C_{\pi} C_{exp}^{\prime}$, }the two Petrov-Galerkin coarse problems \eqref{eqn:localizedprob} and \eqref{eqn:localizedprobadj} for 
    $Q_{0,m}$ and $Q_{0,m}^{\star}$ are both well-defined.
\end{lemma}
\begin{proof} 
It suffices to show the sequilinear form $a(\cdot, \cdot)$ meet an inf-sup condition on the pair  
$V_{H,m}$ and $V_{H,m}^{\star}$. In fact, for every $u_{H,m} \in V_{H,m}$, we show the existence 
a $v_{H,m}^{\star}\in V_{H,m}^{\star}$ such that 
    \begin{equation} \label{eqn:priorconst}
        \frac{|a(u_{H,m},v_{H,m}^{\star})|}{\norm{u_{H,m}}_{1,\kappa}\norm{v_{H,m}^{\star}}_{1,\kappa}}
        \BowenRevise{\geq \frac{1}{2(1+\kappa C_{st}^{\prime})} }\,.
    \end{equation}
To do so, we first show the inf-sup condition \eqref{eqn:priorconst} for the idealized case, that is, 
for every $u_{H,\infty}\in V_{H,\infty}$, there exists $v_{H,\infty}^{\star}\in V_{H,\infty}^{\star}$ such that 
\begin{equation} \label{eqn:prioridealconst}
        \frac{|a(u_{H,\infty},v_{H,\infty}^{\star})|}{\norm{u_{H,\infty}}_{1,\kappa}\norm{v_{H,\infty}^{\star}}_{1,\kappa}}
        \BowenRevise{ \geq \frac{1}{1+\kappa C_{st}^{\prime}} } \,.
\end{equation}
    To see this, for each  $u_{H,\infty} \in V_{H,\infty}$, we know from \eqref{eqn:inf_sup_Vh} 
    the existence of a $v_h \in V_h$ such that 
    \begin{equation*}
        \abs{a(u_{H,\infty},v_h)} \BowenRevise{\geq \frac{1}{1+\kappa C_{st}^{\prime} } } \norm{u_{H,\infty}}_{1,\kappa}\norm{v_h}_{1,\kappa}\,.
    \end{equation*}
We further notice from the orthogonality \eqref{eqn:1} that 
    \begin{equation*}
        a(u_{H,\infty},v_h) = a(u_{H,\infty}, v_h-\mathcal{C}_{\infty}^{\star}v_h).
    \end{equation*}
    By the definition of $I_{H}$ and $\pi_{H}$, $I_{H}(I-\pi_{H})v_h =0$, hence $(I-\mathcal{C}^{\star}_{\infty})(I-\pi_{H})v_h =0$, or 
    \begin{equation*}
        (I-\mathcal{C}^{\star}_{\infty})v_h = (I-\mathcal{C}^{\star}_{\infty})\pi_{H}v_h\,.
    \end{equation*}
    This implies $v_h-\mathcal{C}^{\star}_{\infty}v_h \in V_{H,\infty}^{\star}$. Therefore  
   $v_h-\mathcal{C}^{\star}_{\infty}v_h$ can serve as $v_{H,\infty}^{\star}$ in \eqref{eqn:prioridealconst}.
    
    Now we come to establish \eqref{eqn:priorconst}. 
    For any $u_{H,m} \in V_{H,m}$, we take $u_{H,\infty}=(I-\mathcal{C}_{\infty})\pi_{H}u_{H,m} \in V_{H,\infty}$. 
    For this $u_{H,\infty}\in V_{H,\infty}$, we know from \eqref{eqn:prioridealconst} the existence of 
    a $v_{H,\infty}^{\star}\in V_{H,\infty}^{\star}$ such that \eqref{eqn:prioridealconst} holds. 
    Then we show that $v_{H,m}^{\star}: = (I-\mathcal{C}_{m}^{\star})\pi_{H}v_{H,\infty}^{\star}$ is the desired 
    element to meet \eqref{eqn:priorconst}. To check this, we know by Lemma \ref{lem:LA2} that 
    \begin{equation}\label{eqn:uHmvHm}
        v_{H,m}^{\star}-v_{H,\infty}^{\star} = (\mathcal{C}_{\infty}^{\star}-\mathcal{C}_{m}^{\star})\pi_{H}v_{H,\infty}^{\star}\in W_h, \quad u_{H,m}-u_{H,\infty}=(\mathcal{C}_{\infty}-\mathcal{C}_{m})\pi_{H}u_{H,m} \in W_h.
    \end{equation}
    But it follows from \eqref{eqn:expodecay} and \eqref{eqn:uHmvHm} that 
    \begin{equation}\label{eqn:uvexpodecay}
        \norm{u_{H,m}-u_{H,\infty}}_{1,\kappa} \BowenRevise{\leq C_{exp}^{\prime} C_{\pi} \beta^{m} } \norm{u_{H,m}}_{1,\kappa},\quad \norm{v_{H,m}^{\star}-v_{H,\infty}^{\star}}_{1,\kappa} \BowenRevise{\leq C_{exp}^{\prime} C_{\pi} \beta^{m} } \norm{v_{H,\infty}^{\star}}_{1,\kappa},
    \end{equation}
    and thus
    \begin{equation}\label{eqn:vHinfvHm}
        \norm{v_{H,\infty}^{\star}}_{1,\kappa} \BowenRevise{\geq \frac{1}{1 +  C_{exp}^{\prime} C_{\pi} \beta^{m}} } \norm{v_{H,m}^{\star}}_{1,\kappa}, \ \BowenRevise{\norm{u_{H,\infty}}_{1,\kappa} \geq (1-C_{exp}^{\prime} C_{\pi} \beta^{m} ) \norm{u_{H,m}}_{1,\kappa} }.
    \end{equation}
    Furthermore, we can write by means of the orthogonality \eqref{eqn:1} and \eqref{eqn:uHmvHm} that 
   \begin{equation}\label{eqn:uHmexpan}
        \begin{aligned}
            a(u_{H,m},v_{H,m}^{\star}) &= a(u_{H,m}-u_{H,\infty},v_{H,m}^{\star}) + a(u_{H,\infty},v_{H,m}^{\star}) \\
            & = a(u_{H,m}-u_{H,\infty},v_{H,m}^{\star}) + a(u_{H,\infty},v_{H,\infty}^{\star}) + a(u_{H,\infty},v_{H,m}^{\star}-v_{H,\infty}^{\star}) \\
            & = a(u_{H,m}-u_{H,\infty},v_{H,m}^{\star}) + a(u_{H,\infty},v_{H,\infty}^{\star}).
        \end{aligned}
    \end{equation} 
    The first term can be estimated directly by the continuity of $a(\cdot, \cdot)$,
    Lemma \ref{lem:expodecay} and \eqref{eqn:uHmvHm}
    \begin{equation*}
        \abs{ a(u_{H,m}-u_{H,\infty},v_{H,m}^*) } \BowenRevise{\leq C_{a}} \norm{(\mathcal{C}_{\infty}-\mathcal{C}_{m})\pi_{H}u_{H,m}}_{1,\kappa}\norm{v_{H,m}^*}_{1,\kappa} \BowenRevise{\leq C_{exp}^{\prime\prime}} \beta^{m}\norm{u_{H,m}}_{1,\kappa}\norm{v_{H,m}^*}_{1,\kappa},
    \end{equation*}
    while the second term can be bounded readily by using \eqref{eqn:prioridealconst},  
    \begin{equation}\label{eq:uHinfty}
        \abs{ a(u_{H,\infty},v_{H,\infty}^{\star}) } \BowenRevise{\geq \frac{1}{1+\kappa C_{st}^{\prime}} } \norm{u_{H,\infty}}_{1,\kappa}\norm{v_{H,\infty}^{\star}}_{1,\kappa},
    \end{equation} 
   Then by the triangle inequality we readily get from \eqref{eqn:uHmexpan}-\eqref{eq:uHinfty} that 
    \begin{equation*}
    \begin{aligned}
        \abs{ a(u_{H,m},v_{H,m}^{\star}) } &\geq \abs{ a(u_{H,\infty},v_{H,\infty}^{\star}) } - \abs{ a(u_{H,m}-u_{H,\infty},v_{H,m}^{\star}) } \\
        & \BowenRevise{\geq \frac{1}{1+\kappa C_{st}^{\prime}}} \norm{u_{H,\infty}}_{1,\kappa}\norm{v_{H,\infty}^{\star}}_{1,\kappa} - \BowenRevise{C_{exp}^{\prime\prime} \beta^{m} } \norm{u_{H,m}}_{1,\kappa}\norm{v_{H,m}^{\star}}_{1,\kappa},
    \end{aligned}
    \end{equation*}
    which, along with an application of the triangle inequality and \eqref{eqn:uvexpodecay}-\eqref{eqn:vHinfvHm}, 
    reduces to 
    \begin{equation*}
        \begin{aligned}
             & \quad \ \abs{ a(u_{H,m},v_{H,m}^{\star}) } \\
             & \BowenRevise{\geq \frac{1 - C_{exp}^{\prime} C_{\pi} \beta^{m} }{ (1 + \kappa C_{st}^{\prime})(1+C_{exp}^{\prime}C_{\pi}\beta^{m}) } \norm{u_{H,m}}_{1,\kappa} } 
             \norm{v_{H,m}^{\star}}_{1,\kappa} - \BowenRevise{C_{exp}^{\prime\prime} \beta^{m} } \norm{u_{H,m}}_{1,\kappa}\norm{v_{H,m}^{\star}}_{1,\kappa} \\
             & \BowenRevise{= \frac{ 1 - C_{exp}^{\prime}C_{\pi}\beta^{m} -C_{exp}^{\prime\prime}\beta^{m} (1+\kappa C_{st}^{\prime} ) ( 1 + C_{exp}^{\prime}C_{\pi}\beta^{m} ) }{ (1+\kappa C_{st}^{\prime} ) ( 1 + C_{exp}^{\prime}C_{\pi}\beta^{m} ) } } \norm{u_{H,m}}_{1,\kappa}\norm{v_{H,m}^{\star}}_{1,\kappa}.
        \end{aligned}
    \end{equation*}
    Now \eqref{eqn:priorconst} follows from  \BowenRevise{$ (1+\kappa C_{st}^{\prime})\beta^{m} ( 1 + C_{exp}^{\prime}C_{\pi}\beta^{m} ) C_{exp}^{\prime\prime} \leq \frac{1}{8}$ and $C_{exp}^{\prime} C_{\pi} \beta^{m} \leq \frac{1}{8}$ } 
    induced by Assumption\,\ref{asm:2}.
\end{proof}

\subsection{Localized error estimates}\label{subsec:localization}
In this subsection, we discuss some important approximation properties 
of two localized operators $Q_{0,m}$ and $Q_{0,m}^{\star}$ that are needed in the subsequent 
analysis of two new preconditioners. For the sake of exposition, we first introduce three error functions.  
For each $v_h \in V_h$, we set 
\begin{equation}\label{eq:error_funcs}
e_{h,m}: = (I-Q_{0,m})v_h, 
\q e^{\star}_{h,m}: = (I-Q_{0,m}^{\star})v_h,  
\q 
e_{H,m}: = (I-\mathcal{C}_{m})\pi_{H} e_{h,m}.
\end{equation}
By the definitions of  $\pi_{H}$, $I_{H}$ and $\mathcal{C}_{m}$, we see 
\begin{equation}\label{eqn:projWh}
    I_{H}(e_{h,m} - e_{H,m}) = I_{H}e_{h,m} - I_{H}\pi_{H}e_{h,m} = 0,
\end{equation}
which implies $e_{h,m} - e_{H,m} \in W_h$.

\begin{lemma}\label{lem:actualL2}
Under the condition \eqref{eqn:resolcond} and Assumption \ref{asm:2}, it holds for all $v_h \in V_h$ \BowenRevise{and for $C_{m} = 3 C_{a} ( 1 + C_{exp}^{\prime} C_{\pi} \beta^{m} )$, $C_{m}^{\prime}=(2 C_{exp}^{\prime\prime}(1+\kappa C_{st}^{\prime})\beta^{m}+1)C_{m}$}, that 
    \begin{align}
        \norm{v_h - Q_{0,m}v_h}_{1,\kappa} &\BowenRevise{\leq C_{m}^{\prime}} \norm{v_h}_{1,\kappa}, \label{eqn:actualE-1}\\
        \norm{v_h - Q_{0,m}^{\star}v_h}_{1,\kappa} &\BowenRevise{\leq C_{m}^{\prime}} \norm{v_h}_{1,\kappa}, \label{eqn:actualE-2} \\
        \norm{(I-\mathcal{C}_{m})\pi_{H} e_{h,m} }_{1,\kappa} &\BowenRevise{\leq 2 C_{exp}^{\prime\prime}C_{m}(1+\kappa C_{st}^{\prime})\beta^{m} }\norm{v_h}_{1,\kappa}, \label{eqn:actualE-3}\\
        \norm{Q_{0,m}^{T}v_h}_{1,\kappa} &\BowenRevise{ \leq (1+C_{m}^{\prime}) } \norm{v_h}_{1,\kappa}, \label{eqn:actualEadj}\\
        \norm{v_h - Q_{0,m}v_h}_{L^2(\Omega)} & \BowenRevise{\leq C_{m} (C_{I} n_oH + 2 C_{exp}^{\prime\prime} (\kappa^{-1}+C_{st}^{\prime}) \beta^{m} ) } \norm{v_h}_{1,\kappa},\label{eqn:actualL2-1} \\
        \norm{v_h - Q_{0,m}^{\star}v_h}_{L^2(\Omega)} & \BowenRevise{\leq C_{m} (C_{I} n_oH + 2 C_{exp}^{\prime\prime} (\kappa^{-1}+C_{st}^{\prime}) \beta^{m} ) }  \norm{v_h}_{1,\kappa}. \label{eqn:actualL2-2}
    \end{align}
    \BowenRevise{Under condition \eqref{eqn:oversmplwellpose},  $C_{m}^{\prime} \leq 6 C_{a} (1+ C_{exp}^{\prime}) C_{\pi}$, which is $\kappa$-independent. }
\end{lemma}
\begin{proof}   
    We show only \eqref{eqn:actualE-1}, \eqref{eqn:actualE-3}, \eqref{eqn:actualEadj} and 
    \eqref{eqn:actualL2-1}, while the error estimates \eqref{eqn:actualE-2} and \eqref{eqn:actualL2-2} 
    can be carried out in a similar manner to the ones of \eqref{eqn:actualE-1} and  \eqref{eqn:actualL2-1}  respectively. 
    Before showing these results, we first present a crucial auxiliary error estimate:
    \begin{equation} \label{eqn:triangle1}
        \norm{e_{H,m}}_{1,\kappa} \BowenRevise{\leq 2 C_{a} (1+\kappa C_{st}^{\prime})C_{exp}^{\prime} C_{\pi} \beta^{m}} \norm{e_{h,m} - e_{H,m}}_{1,\kappa}\,.
    \end{equation}    
We derive this by Nitsche's argument, i.e., we introduce $z_{H,m}^{\star}\in V_{H,m}^{\star}$ satisfying 
    \begin{equation}\label{eqn:Nitsche2}
        a_{}(v_{H,m},z_{H,m}^{\star}) = (\nabla v_{H,m}, \nabla e_{H,m}) + \kappa^2 (v_{H,m},e_{H,m})\,,  \,\, 
        \text{ for all } v_{H,m} \in V_{H,m}\,,
    \end{equation}
    whose well-posedness is directly from the inf-sup condition \eqref{eqn:priorconst}.
    Then setting $z_{H,\infty}^{\star}:= (I-\mathcal{C}^{\star}_{\infty})\pi_{H} z_{H,m}^{\star}$, we 
    readily get from \eqref{eqn:trick} that 
    \begin{equation} \label{eqn:Vhmstartrick1}
        z_{H,m}^{\star}-z_{H,\infty}^{\star} = (I-\mathcal{C}_{m}^{\star})\pi_{H}z_{H,m}^{\star}-(I-\mathcal{C}_{\infty}^{\star} )\pi_{H}z_{H,m}^{\star} = (\mathcal{C}_{\infty}^{\star} -\mathcal{C}_{m}^{\star})\pi_{H}z_{H,m}^{\star}.
    \end{equation}
    
    On the other hand, we know from \eqref{eqn:projWh} and the orthogonality \eqref{eqn:1} that 
    \begin{equation}\label{eqn:ehmaortho}
        a(e_{h,m}-e_{H,m},v_{H,\infty}^{\star})=0, \,\,\text{for all } v_{H,\infty}^{\star}\in V_{H,\infty}^{\star}\,.
    \end{equation}
Now choosing $v_{H,m}=e_{H,m}$ in \eqref{eqn:Nitsche2},  we can derive by using \eqref{eqn:localizedprob}, \eqref{eqn:Vhmstartrick1}-\eqref{eqn:ehmaortho} and \eqref{eqn:acont} that 
    \begin{equation} \label{eq:eHm}
        \begin{aligned}
            \norm{e_{H,m}}_{1,\kappa}^{2} = & a_{}(e_{H,m},z_{H,m}^{\star})  
            = a (e_{H,m} - e_{h,m} , z_{H,m}^{\star} ) \\
            = & a_{}(e_{H,m} - e_{h,m} , z_{H,m}^{\star} - z_{H,\infty}^{\star} ) 
            = a_{}(e_{H,m} - e_{h,m}, (\mathcal{C}^{\star}_{\infty}- \mathcal{C}^{\star}_m)\pi_{H} z_{H,m}^{\star}) \\
            \BowenRevise{\leq} & \BowenRevise{C_{a}} \norm{e_{H,m}-e_{h,m}}_{1,\kappa} \norm{(\mathcal{C}^{\star}_{\infty} - \mathcal{C}^{\star}_{m})\pi_{H} z_{H,m}^{\star}}_{1,\kappa}\,.
        \end{aligned}
    \end{equation}
    But by the exponential decay \eqref{eqn:expodecayadj} and stability 
    \eqref{eqn:L2stable}, along with \eqref{eqn:Nitsche2} and \eqref{eqn:priorconst}, we deduce 
    \begin{equation*}
        \norm{(\mathcal{C}^{\star}_{\infty} - \mathcal{C}^{\star}_{m})\pi_{H} z_{H,m}^{\star}}_{1,\kappa} \BowenRevise{\leq C_{exp}^{\prime}C_{\pi}} \beta^{m} \norm{z_{H,m}^{\star}}_{1,\kappa} \BowenRevise{\leq 2 (1+\kappa C_{st}^{\prime})C_{exp}^{\prime} C_{\pi} \beta^{m}}  \norm{e_{H,m}}_{1,\kappa},
    \end{equation*}
    from which and \eqref{eq:eHm} the estimate \eqref{eqn:triangle1} comes immediately. 

    Next, we derive \eqref{eqn:actualE-1} and \eqref{eqn:actualE-3}.
    Noting $e_{h,m} - e_{H,m} \in W_h$, we get from \eqref{eqn:trick} and the definitions of $e_{h,m}$ and $e_{H,m}$ that 
    \begin{equation}\label{eqn:ehmtovh} 
        e_{h,m}-e_{H,m} = v_h - Q_{0,m} v_h - (I-\mathcal{C}_{m})\pi_{H}(v_h - Q_{0,m}v_h) = v_h - (I-\mathcal{C}_m)\pi_{H}v_h.
    \end{equation}
   The orthogonality \eqref{eqn:1} implies
    \begin{equation}\label{eqn:temp1}
        a((I-\mathcal{C}_\infty)\pi_{H}v_h, e_{h,m}-e_{H,m}) =0.
    \end{equation}
  Using this, the coercivity \eqref{eqn:coerWh} of $a({\cdot, \cdot})$ in $W_h$, and \eqref{eqn:acont}, 
  \eqref{eqn:ehmtovh}, \eqref{eqn:expodecay} and \eqref{eqn:L2stable} we arrive at  
    \begin{equation}\label{eqn:projWhest}
        \begin{aligned}
            \norm{e_{h,m} - e_{H,m}}_{1,\kappa}^2 \BowenRevise{\leq} & \  \BowenRevise{3}\Re a_{}(e_{h,m} - e_{H,m},e_{h,m} - e_{H,m}) \\ 
            = & \ \BowenRevise{3}\Re a_{}(v_h, e_{h,m} - e_{H,m}) - \BowenRevise{3}\Re a_{}((I-\mathcal{C}_m) \pi_{H} v_h, e_{h,m} - e_{H,m} ) \\
            = & \ \BowenRevise{3}\Re a_{} (v_h, e_{h,m} - e_{H,m}) - \BowenRevise{3}\Re a_{}((\mathcal{C}_{\infty}-\mathcal{C}_{m})\pi_{H} v_h, e_{h,m} - e_{H,m}) \\
            \BowenRevise{\leq} & \ \BowenRevise{3 C_{a}}\norm{v_h}_{1,\kappa} \norm{e_{h,m}-e_{H,m}}_{1,\kappa} + \BowenRevise{3 C_{a} 
 C_{exp}^{\prime} \beta^{m}} \norm{\pi_{H} v_h}_{1,\kappa} \norm{e_{h,m} - e_{H,m}}_{1,\kappa} \\
            \BowenRevise{\leq} & \ \BowenRevise{3 C_{a} ( 1 + C_{exp}^{\prime} C_{\pi} \beta^{m} ) }\norm{v_h}_{1,\kappa} \norm{e_{h,m} - e_{H,m}}_{1,\kappa}.
        \end{aligned}
    \end{equation}
    Using this and \eqref{eqn:triangle1}, we readily get \eqref{eqn:actualE-1} and \eqref{eqn:actualE-3} 
    by the triangle inequality. 
    
    The continuity of $Q_{0,m}^{T}$ in \eqref{eqn:actualEadj} 
    comes directly from the continuity of $Q_{0,m}$ in \eqref{eqn:actualE-1}.
    In fact, 
    \begin{equation*} 
            \norm{Q_{0,m}^{T}v_h}^{2}_{1,\kappa}  
             = \left( Q_{0,m} Q_{0,m}^{T}v_h, v_h \right)_{1,\kappa} \BowenRevise{\leq (1+C_{m}^{\prime})} \norm{v_h}_{1,\kappa} 
             \norm{Q_{0,m}^T v_h}_{1,\kappa},
    \end{equation*}
   then we get \eqref{eqn:actualEadj} by eliminating $\norm{Q_{0,m}^{T}v_h}_{1,\kappa}$ from both sides.

    We now come to prove \eqref{eqn:actualL2-1}. We first have the triangle inequality
    \begin{equation*}
        \norm{e_{h,m}}_{L^2(\Omega)} \leq \norm{e_{H,m}}_{L^2(\Omega)} + \norm{e_{h,m}-e_{H,m}}_{L^2(\Omega)}.
    \end{equation*}
    But noting $e_{h,m}-e_{H,m} \in W_h$, we can apply \eqref{eqn:IH} and \eqref{eqn:triangle1} to get 
    \begin{equation}\label{eqn:projL2}
    \begin{aligned} 
        \norm{e_{h,m}-e_{H,m}}_{L^2(\Omega)} &\BowenRevise{\leq C_{I} n_o} H \seminorm{e_{h,m}-e_{H,m}}_{H^1(\Omega)}, \\
        \norm{e_{H,m}}_{L^2(\Omega)} &\BowenRevise{\leq 2  C_{exp}^{\prime\prime} (\kappa^{-1} + C_{st}^{\prime}) \beta^{m}}  \norm{e_{h,m}-e_{H,m}}_{1,\kappa}\,.
    \end{aligned}
    \end{equation}
    Now \eqref{eqn:actualL2-1} comes from \eqref{eqn:projWhest} and \eqref{eqn:projL2}. 
     \end{proof}

\begin{lemma}\label{lem:res1loc}
Under the condition \eqref{eqn:resolcond} and Assumption \ref{asm:2}, there holds for all $v_h \in V_h$ that 
    \begin{equation} 
        R_1(v_h): = \abs{ (v_h- e_{h,m},e_{h,m})_{1,\kappa} } \BowenRevise{\lesssim \big(\beta^{m} + \kappa H + \kappa C_{st}^{\prime} \beta^{m} + (\kappa H + \kappa C_{st}^{\prime}\beta^{m})^{1/2} \big) }  \norm{v_h}_{1,\kappa}^{2}\,. \label{eqn:residual1-1loc}
    \end{equation}
\end{lemma}
\begin{proof}
By the definitions of $(\cdot,\cdot)_{1,\kappa}$ and the sesquilinear form $a(\cdot,\cdot)$, we can write 
{\small \begin{equation} \label{eqn:R1locfirst}
\begin{aligned}
    & \q (v_h-e_{h,m},e_{h,m})_{1,\kappa} =  (Q_{0,m}v_h,v_h-Q_{0,m}v_h)_{1,\kappa} \\
    = & \q a(Q_{0,m}v_h,v_h-Q_{0,m}v_h) + 2\kappa^2(Q_{0,m}v_h,v_h-Q_{0,m}v_h)_{L^2(\Omega)} + \imagunit\kappa(Q_{0,m}v_h,v_h-Q_{0,m}v_h)_{L^2(\Gamma)}. 
\end{aligned}
\end{equation} }
Using \eqref{eqn:trick}, we can write the first term on the right-hand side of \eqref{eqn:R1locfirst} as 
\begin{equation}\label{eqn:R1locsecond}
    \begin{aligned}
         a(Q_{0,m}v_h,e_{h,m})
         = a((\mathcal{C}_{\infty}-\mathcal{C}_{m})\pi_{H}Q_{0,m}v_h, e_{h,m}) + a((I-\mathcal{C}_{\infty})\pi_{H}Q_{0,m}v_h,e_{h,m})\,, 
    \end{aligned}
\end{equation}
where the first term can be readily estimated by using \eqref{eqn:acont}, \eqref{eqn:expodecay} and \eqref{eqn:actualE-1}:
\begin{equation}\label{eq:ccm}
\begin{aligned}
   & \,\, \abs{a((\mathcal{C}_{\infty}-\mathcal{C}_{m})\pi_{H}Q_{0,m}v_h, e_{h,m})} \\
    \BowenRevise{\leq C_{a}} & \,\, \norm{(\mathcal{C}_{\infty}-\mathcal{C}_{m})\pi_{H}Q_{0,m}v_h}_{1,\kappa}\norm{e_{h,m}}_{1,\kappa} \BowenRevise{\leq C_{exp}^{\prime\prime} (C_{m}^{\prime}+1)C_{m}^{\prime} \beta^{m}} \norm{v_h}_{1,\kappa}^{2}.
\end{aligned}
\end{equation}
To estimate the second term in 
\eqref{eqn:R1locsecond}, we notice that $(I-\mathcal{C}_{\infty})\pi_{H}Q_{0,m}v_h\in V_{H,\infty}$, 
$e_{h,m}-e_{H,m} \in W_h$, and then can estimate by using 
\eqref{eqn:acont}, \eqref{eqn:actualE-1}, \eqref{eqn:StabWh}, \eqref{eqn:actualE-3} and  \eqref{eqn:1}
to obtain 
\begin{equation} \label{eq:ICinf}
    \begin{aligned}
        \abs{a((I-\mathcal{C}_{\infty})\pi_{H}Q_{0,m}v_h,e_{h,m})} &= 
        \abs{a((I-\mathcal{C}_{\infty})\pi_{H}Q_{0,m}v_h,e_{h,m}-e_{H,m}+e_{H,m})} \\
        & = \abs{a((I-\mathcal{C}_{\infty})\pi_{H}Q_{0,m}v_h,e_{H,m})} \\
        & \BowenRevise{\leq C_{a}} \norm{ (I-\mathcal{C}_{\infty})\pi_{H}Q_{0,m}v_h}_{1,\kappa}\norm{e_{H,m}}_{1,\kappa}\\ 
        & \BowenRevise{\leq 2 C_{a} (C_{\mathcal{C}}+1) C_{\pi} C_{m} (1 + C_{m}^{\prime}) C_{exp}^{\prime\prime} (1+\kappa C_{st}^{\prime})\beta^{m} }\norm{v_h}_{1,\kappa}^{2}.
    \end{aligned}
\end{equation}
The last two terms in \eqref{eqn:R1locfirst} can be bounded readily 
by means of \eqref{eqn:actualL2-1} and \eqref{eqn:actualE-1} to get 
\begin{equation*}
    \kappa^{2}\abs{ (Q_{0,m}v_h,v_h-Q_{0,m}v_h)_{L^2(\Omega)} } \BowenRevise{\leq (1+C_{m}^{\prime}) (C_{I} n_o \kappa H + 2 C_{exp}^{\prime\prime} (1 + \kappa C_{st}^{\prime}) \beta^{m} )} \norm{v_h}_{1,\kappa}^{2},
\end{equation*}
and by means of \eqref{eqn:multr}, \eqref{eqn:actualL2-1} and \eqref{eqn:actualE-1} to obtain 
\begin{equation*}
\begin{aligned}
    & \quad \ \kappa \abs{ (Q_{0,m}v_h,v_h-Q_{0,m}v_h)_{L^2(\Gamma)} } \\
    & \BowenRevise{\leq C_{M}} \kappa^{1/2}\norm{Q_{0,m}v_h}_{L^2(\Omega)}^{1/2}\norm{Q_{0,m}v_h}_{H^1(\Omega)}^{1/2}\kappa^{1/2}\norm{e_{h,m}}_{L^2(\Omega)}^{1/2}\norm{e_{h,m}}_{H^1(\Omega)}^{1/2} \\
    & \BowenRevise{\leq C_{M}} \norm{Q_{0,m}v_h}_{1,\kappa}\kappa^{1/2}\norm{e_{h,m}}_{L^2(\Omega)}^{1/2}\norm{e_{h,m}}_{1,\kappa}^{1/2} \\
    & \BowenRevise{\leq C_{M} C_{m}^{1/2} (1+C_{m}^{\prime})  (C_{m}^{\prime})^{1/2} (C_{I} n_o \kappa H + 2 C_{exp}^{\prime\prime} (1 + \kappa C_{st}^{\prime}) \beta^{m} )^{1/2}} \norm{v_h}_{1,\kappa}^{2}.
\end{aligned}
\end{equation*}
Then the desired estimate follows directly from \eqref{eqn:R1locfirst} and the above 
four estimates, \BowenRevise{where we avoid to write the explicit expression of all constants, but that is clear in the proof.}
\end{proof}

\section{Two hybrid Schwarz preconditioners}\label{sec:hybriddd}
In this section, we are ready to construct two new hybrid Schwarz preconditioners 
and analyse their conditioning properties. 
 
\subsection{Dirichlet type preconditioner}\label{subsec:Dtype}
The first hybrid Schwarz preconditioner that we propose is a Dirichlet type one, where 
each local solver adopts the Dirichlet boundary condition. More specifically,  
for each subdomain $\Om_\ell$, we define a local solution operator $Q_{\ell}: V_h \to \tilde{V}_{h,\ell}$
such that for each $v_h\in V_h$, we seek $Q_\ell v_h\in \tilde{V}_{h,\ell}$ by solving 
\begin{equation} \label{eqn:defQl}
    a_{\ell}(Q_{\ell}v_h,w_{h,\ell}) = a(v_h,w_{h,\ell}), \ \text{for all } w_{h,\ell} \in \tilde{V}_{h,\ell},
\end{equation}
The well-posedness of $Q_{\ell}$ follows from Lemma \ref{lem:ddcontcoer} and the Lax-Milgram theorem. 

Now we can define the first two-level hybrid Schwarz preconditioner: 
\begin{equation}\label{eqn:localprecond1}
    Q_{m}^{(1)} = Q_{0,m} +  (I-Q_{0,m})^{T} \sum\limits_{\ell=1}^{N} Q_{\ell} (I-Q_{0,m}),
\end{equation}
where the transpose $(\cdot)^{T}$ is viewed as the adjoint operator under the $(\cdot,\cdot)_{1,\kappa}-$inner product.

We shall analyse the preconditioner $Q_{m}^{(1)}$ under the following additional conditions. 
\begin{assum}\label{asm:1.1}
    \begin{minipage}[t]{20cm}
    \vspace{-8pt}
        \begin{AssumpList}[label=(\emph{\arabic*}), ref=(\emph{\arabic*}), align=left] 
        \item[] $\kappa H_{\mathrm{sub}}\BowenRevise{\leq C_{\mathrm{Dsub}}}$; \q $H \BowenRevise{\leq c_{0}} \delta$, \label{itm:dd} \\
        \end{AssumpList}
    \end{minipage}
    \BowenRevise{here $C_{\mathrm{Dsub}}\leq \frac{1}{\sqrt{2}}C_{F}^{-1}$. If other restriction of $C_{\mathrm{Dsub}}$ is required,  we will provide the explicit expression of it accordingly. Moreover, $c_{0}$ is a bounded constant.}
    
\end{assum}

\subsubsection{A priori estimates in the subdomains}\label{subsubsec:Qm1}
We now present several a priori estimates that are needed later for the analysis of the 
preconditioner $Q_{m}^{(1)}$.

\begin{lemma}\label{lem:stablelocsol} 
    Under Assumption \ref{asm:1.1}, for all $v_h \in V_h$, we have  
    \begin{equation} \label{eqn:upddbnd}
\sum\limits_{\ell}\norm{Q_{\ell}v_h}_{1,\kappa,\Omega_{\ell}}^{2} \BowenRevise{\leq 9 (C_{a}^{\prime\prime})^{2} } \Lambda \norm{v_h}_{1,\kappa}^2.
    \end{equation}
\end{lemma}
\begin{proof}
We know from \eqref{eqn:contaesc} that 
\begin{equation}\label{eqn:contdd2}
    \abs{a_{\ell}(v_h, Q_{\ell}v_h)} \BowenRevise{\leq C_{a}^{\prime\prime}} \norm{v_h}_{1,\kappa,\Omega_{\ell}}\norm{Q_{\ell}v_h}_{1,\kappa,\Omega_{\ell}}.
\end{equation}
    Noticing that $ a_{\ell} (v_h,Q_{\ell}v_h)=a(v_h,Q_{\ell}v_h)$, we further derive 
    by the local coercivity \eqref{eqn:coeral} and the continuity \eqref{eqn:contdd2} that 
    \begin{equation*}
        \begin{aligned}
        \sum\limits_{\ell}\norm{Q_{\ell}v_h}_{1,\kappa,\Omega_{\ell}}^{2} & \BowenRevise{\leq 3} 
            \sum\limits_{\ell} \Re a_{\ell}(Q_{\ell}v_h, Q_{\ell}v_h) \leq \BowenRevise{3} \sum\limits_{\ell}\abs{ a_{\ell}(v_h, Q_{\ell}v_h)} \BowenRevise{\leq 3 C_{a}^{\prime\prime}} \sum\limits_{\ell} \norm{v_h}_{1,\kappa,\Omega_{\ell}}\norm{Q_{\ell}v_h}_{1,\kappa,\Omega_{\ell}} \\
            & \BowenRevise{\leq 3 C_{a}^{\prime\prime}} \Big(\sum\limits_{\ell}\norm{v_h}_{1,\kappa,\Omega_{\ell}}^{2}\Big)^{1/2}\Big( \sum\limits_{\ell}\norm{Q_{\ell}v_h}_{1,\kappa,\Omega_{\ell}}^{2} \Big)^{1/2}.
        \end{aligned}
    \end{equation*}
    Now \eqref{eqn:upddbnd} follows from the finite overlap property \eqref{eqn:finiteoverlap}.
\end{proof}

\begin{lemma}\label{lem:dd2} 
    There holds for all $w\in W_h$ that 
    \begin{equation} \label{eqn:energydd}
        \sum\limits_{\ell=1}^{N} \norm{\Pi_h\chi_{\ell}w}_{1,\kappa,\Omega_{\ell}}^2 \BowenRevise{\leq}  C(\Lambda,H,\delta) \norm{w}_{1,\kappa}^{2}\,, \q \text{with } ~~C(\Lambda,H,\delta) = \BowenRevise{2 C_{p}^{2}C_{I}^{2}n_oC_{\Pi}^{2}}\Lambda \Big( 1 + (\frac{H}{\delta})^2  \Big).
    \end{equation}
\end{lemma}
\begin{proof}
    By Lemma \ref{lem:nodint} and the triangle inequality we can deduce 
    \begin{equation*}
    \begin{aligned}
        \sum\limits_{\ell=1}^{N} \norm{\Pi_h \chi_{\ell}w}_{1,\kappa,\Omega_{\ell}}^{2} & \BowenRevise{\leq} \sum\limits_{\ell=1}^{N}\Big(\norm{\chi_{\ell}w}_{1,\kappa,\Omega_{\ell}}^{2} + \BowenRevise{C_{\Pi}^{2} (1+\kappa h_{\ell})^{2} }(\frac{h_{\ell}}{\delta_{\ell}})^{2}\norm{w}_{1,\kappa,\Omega_{\ell}}^{2} \Big)\\
        & \BowenRevise{\leq } \sum\limits_{\ell=1}^{N}\norm{\chi_{\ell}w}_{1,\kappa,\Omega_{\ell}}^{2} + \Lambda \BowenRevise{C_{\Pi}^{2} (1+\kappa h)^{2} } (\frac{h}{\delta})^{2}\norm{w}_{1,\kappa}^{2}.
    \end{aligned}
    \end{equation*}
    By the definition of $\norm{\cdot}_{1,\kappa,\Omega_{\ell}}$ and the property \eqref{eqn:pouprop} 
    of $\chi_{\ell}$, we derive for any $w \in V_h$,
    \begin{equation}\label{eqn:energynormchi}
        \norm{\chi_{\ell}w}_{1,\kappa,\Omega_{\ell}}^{2}\leq \norm{(\nabla\chi_{\ell})w}_{L^2(\Omega_{\ell})}^{2} + \norm{\chi_{\ell}\nabla w}_{L^2(\Omega_{\ell})}^{2} + \kappa^2\norm{\chi_{\ell}w}_{L^2(\Omega)}^{2} \BowenRevise{\leq} \frac{\BowenRevise{C_{p}^{2}} }{\delta^{2}_{\ell}}\norm{w}_{L^2(\Omega_{\ell})}^{2} +  \norm{w}_{1,\kappa,\Omega_{\ell}}^{2}.
    \end{equation}
    But for $w\in W_h$,  we can readily get from \eqref{eqn:IH} that 
    \begin{equation*}
        \sum\limits_{\ell}\frac{1}{\delta_{\ell}^2} \norm{w}_{L^2(\Omega_{\ell})}^2 \BowenRevise{\leq }
        \Lambda \frac{1}{\delta^2}\norm{w}_{L^2(\Omega)}^2=\Lambda\frac{1}{\delta^2}\norm{w-I_{H}w}_{L^2(\Omega)}^{2} \BowenRevise{\leq C_{I}^{2} n_o}
        \Lambda (\frac{H}{\delta})^2 \vert w \vert_{H^1(\Omega)}^{2}.
    \end{equation*}
    Combining the above three estimates, we come to 
    \begin{equation*}
        \sum\limits_{\ell=1}^{N} \norm{\Pi_h \chi_{\ell}w}_{1,\kappa,\Omega_{\ell}}^{2} \BowenRevise{\leq} \Lambda ( 1 + \BowenRevise{ C_{p}^{2} C_{I}^{2} n_o}(\frac{H}{\delta})^2 + \BowenRevise{C_{\Pi}^{2} (1+\kappa h)^{2} }(\frac{h}{\delta})^2 ) \norm{w}^{2}_{1,\kappa}, 
    \end{equation*}
    which implies \eqref{eqn:energydd}, by noting the fact that $h\le H$ \BowenRevise{and $\kappa h \leq 1$}.
\end{proof}


\begin{lemma}\label{lem:solsubdd}
Under the condition \eqref{eqn:resolcond}, there holds for all $w\in W_h$ that 
    \begin{equation}\label{eqn:onelvldecomp}
        \sum\limits_{\ell} \norm{Q_{\ell}w}_{1,\kappa,\Omega_{\ell}}^2 \BowenRevise{\geq \frac{1}{9}}  
        ( C(\Lambda,H,\delta))^{-1} \norm{w}_{1,\kappa}^{2}\,.
    \end{equation}
\end{lemma}
\begin{proof}
For every $w\in W_h$, we know by using the properties of $\chi_{\ell}$ that 
$\chi_{\ell}w$ vanishes on $\Gamma_{\ell}\backslash \Gamma$, 
$\Pi_h\chi_{\ell}w \in \tilde{V}_{h,\ell}$, and 
$w$ coincides with $\sum_{\ell}\chi_{\ell}w$ at every node of $\mathcal{N}^{h}$. This implies the identity that 
\begin{equation*}
    w=\Pi_h w = \Pi_h\sum_{\ell}\chi_{\ell}w = \sum_{\ell}\Pi_h\chi_{\ell}w.
\end{equation*}
Using the coercivity \eqref{eqn:coerWh} of the real part of the sesquilinear form $a(\cdot, \cdot)$ in $W_h$,  
the definition of $Q_{\ell}$ in \eqref{eqn:defQl}, Lemma \ref{lem:dd2} and Cauchy's inequality, 
we derive
\begin{equation} \label{eqn:bi-stableapply}
    \begin{aligned}
        \norm{w}_{1,\kappa}^{2} \BowenRevise{\leq 3} \Re a(w, w) 
        = & \BowenRevise{3}\sum\limits_{\ell}\Re a(w, \Pi_h \chi_{\ell}w)
        = \BowenRevise{3} \sum\limits_{\ell}\Re a_{\ell}(Q_{\ell}w, \Pi_h\chi_{\ell}w) \\
        \leq & \BowenRevise{3} \sum\limits_{\ell}\norm{Q_{\ell}w}_{1,\kappa,\Omega_{\ell}} \norm{\Pi_h\chi_{\ell}w}_{1,\kappa,\Omega_{\ell}} \\
        \leq & \BowenRevise{3} \Big( \sum\limits_{\ell} \norm{Q_{\ell}w}_{1,\kappa,\Omega_{\ell}}^2 \Big)^{1/2} \Big( \sum\limits_{\ell} \norm{\Pi_h \chi_{\ell} w }_{1,\kappa,\Omega_{\ell}}^{2}\Big)^{1/2} \\
        \BowenRevise{\leq} & \BowenRevise{3} (C(\Lambda,H,\delta))^{1/2}\Big( \sum\limits_{\ell} \norm{Q_{\ell}w}_{1,\kappa,\Omega_{\ell}}^2 \Big)^{1/2} \norm{w}_{1,\kappa},
    \end{aligned}
\end{equation}
which implies \eqref{eqn:onelvldecomp} immediately. 
\end{proof}

For any $v_h \in V_h$, we define and then estimate its two error quantities: 
\begin{align}
    & R_{2,1}(v_h): = \sum\limits_{\ell} \kappa^2 \norm{(I-Q_{\ell})(e_{h,m}-e_{H,m})}_{L^2(\Omega_{\ell})} \norm{Q_{\ell}(e_{h,m}-e_{H,m})}_{L^2(\Omega_{\ell})}, \label{eq:R21} \\
    & R_{2,2}(v_h): = \sum\limits_{\ell} \kappa \norm{(I-Q_{\ell})(e_{h,m}-e_{H,m})}_{L^2(\Gamma_{\ell}\cap \Gamma)} \norm{Q_{\ell}(e_{h,m}-e_{H,m})}_{L^2(\Gamma_{\ell}\cap \Gamma)} . \label{eq:R22}
\end{align}

\begin{lemma} \label{lem:res22loc}
 Under Assumptions \ref{asm:1.1}, \ref{asm:2} and \eqref{eqn:resolcond}, the error estimates are true 
for all $v_h \in V_h$: 
    \begin{equation*}
        R_{2,1}(v_h) \BowenRevise{\leq C_{2,1}\Lambda} \kappa H_{\mathrm{sub}} \norm{v_h}_{1,\kappa}^{2}\,, 
        \quad \quad R_{2,2}(v_h) \BowenRevise{\leq C_{2,2}\Lambda} \kappa H_{\mathrm{sub}} \norm{v_h}_{1,\kappa}^{2},
    \end{equation*}
    \BowenRevise{where $C_{2,1}=3 C_{a}^{\prime\prime}(3C_{a}^{\prime\prime}+1) C_{F} C_{m}^{2}$, $C_{2,2} = 3C_{m}^{2}C_{tr}C_{a}^{\prime\prime}c_{0}(3C_{a}^{\prime\prime}C_{tr} + (n_oC_{I}C_{M})^{1/2})$.}
\end{lemma}
\begin{proof}
We first estimate $R_{2,1}(v_h)$. 
It follows directly from the definition of $\norm{\cdot}_{1,\kappa,\Omega_{\ell}}$ that 
\begin{equation*}
    \kappa\norm{(I-Q_{\ell})(e_{h,m}-e_{H,m})}_{L^2(\Omega_{\ell})}\leq \norm{(I-Q_{\ell})(e_{h,m}-e_{H,m})}_{1,\kappa,\Omega_{\ell}}.
\end{equation*}
But by noting that $Q_{\ell} (e_{h,m}-e_{H,m}) \in \tilde{V}_{h,\ell}$, we readily get from \eqref{eqn:trace} that 
\begin{equation*}
    \norm{Q_{\ell}(e_{h,m}-e_{H,m})}_{L^2(\Omega_{\ell})} \BowenRevise{\leq C_{F}} H_{\ell}\seminorm{Q_{\ell}(e_{h,m}-e_{H,m})}_{H^1(\Omega_{\ell})},
\end{equation*}
combing the above two estimates, we come to 
\begin{equation*}
\begin{aligned}
  & \quad\kappa^2\norm{(I-Q_{\ell})(e_{h,m}-e_{H,m})}_{L^2(\Omega_{\ell})}\norm{Q_{\ell}(e_{h,m}-e_{H,m})}_{L^2(\Omega_{\ell})} \\
        & \BowenRevise{\leq C_{F}}  \kappa H_{\mathrm{sub}}\norm{(I-Q_{\ell})(e_{h,m}-e_{H,m})}_{1,\kappa,\Omega_{\ell}}\norm{Q_{\ell}(e_{h,m}-e_{H,m})}_{1,\kappa,\Omega_{\ell}}.
\end{aligned}
\end{equation*}
Summing over all subdomains and using the triangle inequality, \eqref{eqn:upddbnd}, \eqref{eqn:finiteoverlap} and \eqref{eqn:projWhest}, we can directly derive the desired estimate of $R_{2,1}(v_h)$, namely, 
\begin{equation*} 
    \begin{aligned}
        R_{2,1}(v_h) & \BowenRevise{\leq C_{F}} \kappa H_{\mathrm{sub}} \sum\limits_{\ell}(\norm{Q_{\ell}(e_{h,m}-e_{H,m})}_{1,\kappa,\Omega_{\ell}}^{2} + \norm{e_{h,m}-e_{H,m}}_{1,\kappa,\Omega_{\ell}}\norm{Q_{\ell}(e_{h,m}-e_{H,m})}_{1,\kappa,\Omega_{\ell}} ) \\
        & \BowenRevise{\leq C_{F}} \kappa H_{\mathrm{sub}}\big( \BowenRevise{3 \Lambda C_{a}^{\prime\prime}C_{m}^{2} } \norm{v_h}_{1,\kappa}^{2} \!+\! (\sum\limits_{\ell}\norm{e_{h,m}-e_{H,m}}_{1,\kappa,\Omega_{\ell}}^{2})^{\frac{1}{2}} (\sum\limits_{\ell}\norm{Q_{\ell}(e_{h,m}-e_{H,m})}_{1,\kappa,\Omega_{\ell}}^{2})^{\frac{1}{2}}\big) \\
        & \BowenRevise{\leq 3 C_{a}^{\prime\prime}(3C_{a}^{\prime\prime}+1) C_{F} C_{m}^{2} \Lambda} \kappa H_{\mathrm{sub}}  \norm{v_h}_{1,\kappa}^{2}.
    \end{aligned}
    \end{equation*}

   Next, we estimate $R_{2,2}(v_h)$. 
   By the triangle inequality, we have 
    \begin{equation}\label{eqn:res22loc1}
        \begin{aligned}
            & \quad \sum\limits_{\ell}\norm{(I-Q_{\ell})(e_{h,m}-e_{H,m})}_{L^2(\Gamma_{\ell}\cap\Gamma)}\norm{Q_{\ell}(e_{h,m}-e_{H,m})}_{L^2(\Gamma_{\ell}\cap\Gamma)} \\
            & \leq \sum\limits_{\ell}(\norm{e_{h,m}-e_{H,m}}_{L^2(\Gamma_{\ell}\cap\Gamma)}\norm{Q_{\ell}(e_{h,m}-e_{H,m})}_{L^2(\Gamma_{\ell}\cap\Gamma)} + \norm{Q_{\ell}(e_{h,m}-e_{H,m})}^{2}_{L^2(\Gamma_{\ell}\cap\Gamma)}).
        \end{aligned}
    \end{equation}
    Since $Q_{\ell} (e_{h,m}-e_{H,m}) \in \tilde{V}_{h,\ell}$, 
    using \eqref{eqn:traceL2sub2}, \eqref{eqn:upddbnd}, \eqref{eqn:projWhest} and the finite overlap \eqref{eqn:finiteoverlap}, we derive 
    \begin{equation}\label{eqn:res22loc2}
        \sum\limits_{\ell}\norm{Q_{\ell}(e_{h,m}-e_{H,m})}^{2}_{L^2(\Gamma_{\ell}\cap \Gamma)} \BowenRevise{\leq 9 (C_{a}^{\prime\prime} C_{tr})^{2}} \Lambda H_{\mathrm{sub}}\norm{e_{h,m}-e_{H,m}}_{1,\kappa}^2 \BowenRevise{\leq 9 (C_{a}^{\prime\prime} C_{tr} C_{m})^{2}} \Lambda H_{\mathrm{sub}}\norm{v_h}_{1,\kappa}^{2},
    \end{equation}
    On the other hand, we know $e_{h,m}-e_{H,m}\in W_h$ from \eqref{eqn:projWh}. Then we obtain 
    by the Cauchy-Schwarz inequality, \eqref{eqn:traceL2sub1}\BowenRevise{-\eqref{eqn:traceL2sub2}} and \eqref{eqn:res22loc2} that 
    \begin{equation}\label{eqn:res22loc3}
    \begin{aligned}
        &\quad \sum\limits_{\ell}(\norm{e_{h,m}-e_{H,m}}_{L^2(\Gamma_{\ell}\cap\Gamma)}\norm{Q_{\ell}(e_{h,m}-e_{H,m})}_{L^2(\Gamma_{\ell}\cap\Gamma)}) \\
        &\leq (\sum\limits_{\ell}\norm{e_{h,m}-e_{H,m}}_{L^2(\Gamma_{\ell}\cap\Gamma)}^2)^{1/2} (\sum\limits_{\ell}\norm{Q_{\ell}(e_{h,m}-e_{H,m})}_{L^2(\Gamma_{\ell}\cap\Gamma)}^{2})^{1/2}\\
        & \BowenRevise{\leq 3 C_{a}^{\prime\prime}(n_oC_{I} C_{M})^{\frac{1}{2}} C_{tr} } \Lambda  H^{\frac{1}{2}}H^{\frac{1}{2}}_{\mathrm{sub}}\norm{e_{h,m}-e_{H,m}}_{1,\kappa}^{2} 
        \BowenRevise{\leq 3 C_{a}^{\prime\prime}(n_oC_{I} C_{M})^{\frac{1}{2}} C_{tr} C_{m}^{2} c_{0} } \Lambda  H_{\mathrm{sub}}\norm{v_h}_{1,\kappa}^{2}.
    \end{aligned}
    \end{equation}
    Combining \eqref{eqn:res22loc1}-\eqref{eqn:res22loc3}, we conclude the desired estimate of $R_{2,2}(v_h)$.
\end{proof}

\subsubsection{Optimality estimates of the first preconditioner $Q_{m}^{(1)}$}\label{subsubsec:optimal_Qm1}
We are now ready to demonstrate the optimality of the preconditioner $Q_{m}^{(1)}$, namely, 
we establish the upper bound of the energy-norm (Theorem\,\ref{thm:upboundact}) and the lower bound 
of the field of values of $Q_{m}^{(1)}$ (Theorem\,\ref{thm:lowerboundact}), which are proved 
both independent of the key parameters, $\kappa$, $h$, $H$, $\delta$ and $H_\mathrm{{sub}}$. 
\begin{theorem}\label{thm:upboundact}
    Under Assumptions \ref{asm:1.1}, \ref{asm:2} and \eqref{eqn:resolcond}, the following estimate holds 
     \BowenRevise{ 
    \begin{equation}\label{eq:upper}
    \norm{Q_{m}^{(1)} v_h}_{1,\kappa}^2 \leq \Big(  2(1+C_{m}^{\prime})^{2} + 18 
 (C_{a}^{\prime\prime})^{2} (C_{m}^{\prime})^{4} \Lambda^{2} \Big)  \norm{v_h}_{1,\kappa}^2, \, \, {\rm for \, all \,} \,v_h \in V_h,
    \end{equation}
    where  $C_{m}^{\prime}$ is defined in Lemma \ref{lem:actualL2} and $C_{a}^{\prime\prime}$ is defined in Lemma \ref{lem:ddcontcoer}.}
    \end{theorem}
\begin{proof}
    We apply the definition of $Q_{m}^{(1)}$ in \eqref{eqn:localprecond1}, 
    Lemma \ref{lem:actualL2} and \eqref{eqn:upddbnd} to readily derive 
        \begin{equation*}
        \begin{aligned}
            \norm{Q_{m}^{(1)}v_h}_{1,\kappa}^{2} & \leq 2 \Big( \norm{Q_{0,m}v_h}_{1,\kappa}^{2} + \norm{(I-Q_{0,m}^{T})\sum\limits_{\ell}Q_{\ell}e_{h,m}}_{1,\kappa}^{2} \Big) \\
            & \BowenRevise{\leq 2} \Big( \norm{Q_{0,m}v_h}_{1,\kappa}^{2} + \BowenRevise{(C_{m}^{\prime})^{2}}\norm{\sum\limits_{\ell}Q_{\ell}e_{h,m}}_{1,\kappa}^{2} \Big) 
            \\
            & \BowenRevise{\leq} \BowenRevise{2 (1+C_{m}^{\prime})^{2}}\norm{v_h}_{1,\kappa}^{2} + \BowenRevise{2 (C_{m}^{\prime})^{2}}\Lambda \sum\limits_{\ell}
            \norm{Q_{\ell}e_{h,m}}_{1,\kappa,\Omega_{\ell}}^{2} \\
            & \BowenRevise{\leq \Big(  2(1+C_{m}^{\prime})^{2} + 18 
 (C_{a}^{\prime\prime})^{2} (C_{m}^{\prime})^{4} \Lambda^{2} \Big) }  \norm{v_h}_{1,\kappa}^{2}.
        \end{aligned}
    \end{equation*}
This completes the proof of \eqref{eq:upper}. 
\end{proof}

To proceed further, we first give some simple estimates of the coefficients involved in the proof of 
the next theorem. \BowenRevise{Define $c_{1} = \min\{ \Lambda, (18(1+c_{0}^{2})C_{p}^{2}C_{I}^{2}n_oC_{\Pi}^{2})^{-1} \}$, $C_{2} = C_{2,1}+2C_{2,2}$ and  $\tilde{C}_{m}=6 C_{a}^{\prime\prime} C_{m} C_{exp}^{\prime\prime}(3C_{a}(1+C_{exp}^{\prime} C_{\pi} \beta^{m}) + C_{m}^{\prime})$.} 
\BowenRevise{Let 
\begin{equation} \label{eqn:coefR2}
    C(\Lambda,\kappa,H_{\mathrm{sub}},\beta) :=  C_{2} \Lambda \kappa H_{\mathrm{sub}} + \tilde{C}_{m} \Lambda (1+\kappa C_{st}^{\prime})\beta^{m},
\end{equation}} 
\BowenRevise{$C_{1,1}\!\!:=\!(C_{m}^{\prime}\!+\!1)\max\{ C_{exp}^{\prime\prime} C_{m}^{\prime}  (1\!+\!\kappa C_{st}^{\prime})^{-1} ,  2 C_{a} (C_{\mathcal{C}}+1) C_{\pi} C_{m} C_{exp}^{\prime\prime},  4C_{exp}^{\prime\prime} \}$, $C_4\!:=\! C_M^2 C_mC_{m}^{\prime}(1\!+\!C_{m}^{\prime})^2C_{exp}^{\prime\prime}$,
 $\Theta_{1}(x)\!=\!\max\{ 32C_{1,1}x, 4096C_4 x^2\}$ and $\Theta_{2}(x)\!=\!\min\{\frac{x^2}{2048C_4C_In_o}, \frac{x}{64(1+C_{m}^{\prime}) C_In_o} \}$, when 
\begin{equation}\label{eqn:R1m}
    m \geq \abs{\log \beta}^{-1} \log \big((1+\kappa C_{st}^{\prime}) \Theta_{1}(c_1^{-1}\Lambda)\big),\quad  \kappa H \leq \Theta_{2}(c_1^{-1}\Lambda),
\end{equation}
we have $ R_{1}(v_{h})\leq \frac{c_1}{8}\Lambda^{-1}\norm{v_{h}}_{1,\kappa}^{2}$. Moreover, if 
\begin{equation}\label{eqn:SubHm}
    \kappa  H_{\mathrm{sub}} \leq \frac{c_1}{16c_2\Lambda^2}, \quad 
    m \geq \abs{\log \beta}^{-1} \log \big((1+\kappa C_{st}^{\prime})\max\{ 4C_{exp}^{\prime\prime}C_m, 16\tilde{C}_{m}c_1^{-1}\Lambda^2\} \big),
\end{equation}
we have $ C(\Lambda,\kappa,H_{\mathrm{sub}},\beta) \leq \frac{c_1}{8}\Lambda^{-1}\norm{v_{h}}_{1,\kappa}^{2}$ and $\norm{e_{H,m}}_{1,\kappa}^{2}\leq \frac{1}{4}\norm{v_{h}}_{1,\kappa}^{2}$.
}

\BowenRevise{
Next we will use \eqref{eqn:R1m}-\eqref{eqn:SubHm} as the specific assumption for Theorem \ref{thm:lowerboundact}.} 
These specific bounds are needed in the next theorem for deriving a desired lower bound of the field of values 
of the preconditioner $Q_{m}^{(1)}$. 

\begin{theorem}\label{thm:lowerboundact}
    Under Assumptions \ref{asm:1.1}, \ref{asm:2} and \eqref{eqn:resolcond}, the field of values of $Q_{m}^{(1)}$ can be bounded 
    below from zero, namely, 
    \begin{equation} \label{eqn:lowerboundloc}
        \abs{( Q_{m}^{(1)}v_h,v_h )_{1,\kappa}} \BowenRevise{\geq \frac{c_{1}}{4}\Lambda^{-1}} \norm{v_h}_{1,\kappa}^{\BowenRevise{2}}, \q 
        {\rm for \, all } \,\,v_h \in V_h\,. 
    \end{equation}
\end{theorem}
\begin{proof}
First, we expand the field of values by the definition of $Q_{m}^{(1)}$:
\begin{equation}\label{eqn:othersideloc1}
    (v_h, Q_{m}^{(1)} v_h)_{1,\kappa}
        = (v_h, Q_{0,m}v_h)_{1,\kappa} + \sum\limits_{\ell} (e_{h,m},  Q_{\ell}e_{h,m})_{1,\kappa},
\end{equation}
    and for every $\ell$, we can write 
    \begin{equation}\label{eqn:ehmexpansion}
        (e_{h,m},Q_{\ell}e_{h,m})_{1,\kappa} = (e_{h,m},Q_{\ell}e_{h,m})_{1,\kappa,\Omega_{\ell}} =\norm{Q_{\ell}(e_{h,m}-e_{H,m})}^{2}_{1,\kappa,\Omega_{\ell}} + r_{\mathrm{loc},\ell}(v_h),
    \end{equation}
    where $r_{\mathrm{loc},\ell}(v_h)$ is a residual term given by 
    \begin{equation}\label{eqn:rlocl}
        \begin{aligned}
             r_{\mathrm{loc},\ell}(v_h) 
            &:= (e_{h,m}-e_{H,m}-Q_{\ell}(e_{h,m}-e_{H,m}),Q_{\ell}(e_{h,m}-e_{H,m}))_{1,\kappa,\Omega_{\ell}}\\ & \quad +  (e_{H,m},Q_{\ell}(e_{h,m}-e_{H,m}))_{1,\kappa,\Omega_{\ell}} + (e_{h,m},Q_{\ell}e_{H,m})_{1,\kappa,\Omega_{\ell}}.
        \end{aligned}
    \end{equation}
    Substituting \eqref{eqn:ehmexpansion} and \eqref{eqn:rlocl} into \eqref{eqn:othersideloc1}, we obtain by 
    the triangle inequality that 
    \begin{equation} \label{eqn:othersidelocfov2}
        \begin{aligned}
             &\quad \ \abs{(v_h, Q_{m}^{(1)}v_h)_{1,\kappa} }\!  \\
             & = \abs{\norm{Q_{0,m}v_h}_{1,\kappa}^{2} + {(v_h - Q_{0,m}v_h, Q_{0,m}v_h)_{1,\kappa}} + { \sum\limits_{\ell}(e_{h,m},Q_{\ell}e_{h,m})_{1,\kappa} } } \\
             & = \abs{\norm{Q_{0,m}v_h}_{1,\kappa}^{2} + {(v_h - Q_{0,m}v_h, Q_{0,m}v_h)_{1,\kappa}} + \sum\limits_{\ell}\norm{Q_{\ell}(e_{h,m}-e_{H,m})}^{2}_{1,\kappa,\Omega_{\ell}} \\
             & \,\,+   \sum\limits_{\ell} r_{\mathrm{loc,\ell}}(v_h) } \\
             & \geq \abs{\norm{Q_{0,m}v_h}_{1,\kappa}^{2} \! + \!\sum\limits_{\ell}\norm{Q_{\ell}(e_{h,m}-e_{H,m})}^{2}_{1,\kappa,\Omega_{\ell}}}\! -\! R_1(v_h)\! -\! R_{\mathrm{loc}}(v_h),
        \end{aligned}
    \end{equation}
    where the last term $R_{\mathrm{loc}}(v_h)$ is given by 
    \begin{equation}\label{eqn:defresloc}
    \begin{aligned}
        R_{\mathrm{loc}}(v_h)= \sum\limits_{\ell}\abs{
            r_{\mathrm{loc,\ell}}(v_h)}.
    \end{aligned}
    \end{equation}
    
    Next, we estimate all the terms in the lower bound of \eqref{eqn:othersidelocfov2}. 
    First, for the second term in \eqref{eqn:othersidelocfov2}, we notice $e_{h,m}-e_{H,m} \in W_h$, then we can apply 
    Lemma \ref{lem:solsubdd} to get 
    \begin{equation}\label{eqn:lowbndact}
        \sum\limits_{\ell}\norm{Q_{\ell}(e_{h,m}-e_{H,m})}_{1,\kappa,\Omega_{\ell}}^{2} \BowenRevise{\geq \frac{1}{9} } C(\Lambda,H,\delta)^{-1} \norm{e_{h,m} - e_{H,m}}_{1,\kappa}^{2}.
    \end{equation}
    To bound the term $R_{\mathrm{loc}}(v_h)$, it suffices to bound 
    all the terms in \eqref{eqn:rlocl}. To do so, we first apply \eqref{eqn:upddbnd} and \eqref{eqn:actualE-3} to deduce 
    \begin{equation} \label{eqn:resproj1}
        \sum\limits_{\ell} \norm{Q_{\ell}e_{H,m}}_{1,\kappa,\Omega_{\ell}}^{2}\BowenRevise{\leq 9 (C_{a}^{\prime\prime})^{2}} \Lambda \norm{e_{H,m}}_{1,\kappa}^{2} \BowenRevise{\leq 36 ( C_{a}^{\prime\prime} C_{exp}^{\prime\prime}C_{m}(1+\kappa C_{st}^{\prime})\beta^{m})^{2} } \Lambda \norm{v_h}_{1,\kappa}^{2}.
    \end{equation}
    This, along with \eqref{eqn:finiteoverlap}, \eqref{eqn:upddbnd}, \eqref{eqn:actualE-3}, \eqref{eqn:projWhest} and \eqref{eqn:actualE-1}, we can  deduce the bounds of the second and third terms in \eqref{eqn:rlocl}: 
    \begin{equation*}
    \begin{aligned}
        \sum\limits_{\ell} \abs{(e_{H,m},Q_{\ell}(e_{h,m}-e_{H,m}))_{1,\kappa,\Omega_{\ell}}} 
       & \leq (\sum\limits_{\ell}\norm{e_{H,m}}_{1,\kappa,\Omega_{\ell}}^{2})^{1/2} (\sum\limits_{\ell}\norm{Q_{\ell}(e_{h,m}-e_{H,m})}_{1,\kappa,\Omega_{\ell}}^{2})^{1/2} \\
        &  
        \BowenRevise{\leq 18 C_{a}C_{a}^{\prime\prime}\Lambda(1\!+\! C_{exp}^{\prime}C_{\pi}\beta^{m}) C_{exp}^{\prime\prime}C_{m}(1\!+\! \kappa C_{st}^{\prime}) \beta^{m} }  
        \norm{v_h}_{1,\kappa}^{2}, \\
        \sum\limits_{\ell}\abs{(e_{h,m},Q_{\ell}e_{H,m})_{1,\kappa,\Omega_{\ell}}}  & \leq (\sum\limits_{\ell}\norm{e_{h,m}}_{1,\kappa,\Omega_{\ell}}^{2})^{1/2} (\sum\limits_{\ell}\norm{Q_{\ell}e_{H,m}}_{1,\kappa,\Omega_{\ell}}^{2})^{1/2} \\
            & 
            \BowenRevise{\leq 6 C_{a}^{\prime\prime}\Lambda C_{exp}^{\prime\prime}C_{m}(1+\kappa C_{st}^{\prime})\beta^{m} C_{m}^{\prime} }  \norm{v_h}_{1,\kappa}^{2}\,.
    \end{aligned}
    \end{equation*}
    These two bounds enable us to get an upper estimate of $R_{\mathrm{loc}}(v_h)$: 
    \begin{equation*} \label{eqn:residualloc}
        R_{\mathrm{loc}}(v_h) \BowenRevise{\leq \tilde{C}_{m} \Lambda (1+\kappa C_{st}^{\prime})\beta^{m}}  
        \norm{v_h}_{1,\kappa}^2 + \sum\limits_{\ell}\abs{((I-Q_{\ell})(e_{h,m}-e_{H,m}),Q_{\ell}(e_{h,m}-e_{H,m}))_{1,\kappa,\Omega_{\ell}}}.
    \end{equation*}
To bound the second summation term above, we rewrite its summands as 
\begin{equation}\label{eqn:ddresidualloc}
\begin{aligned}
        & \quad\  ((I-Q_{\ell})(e_{h,m}-e_{H,m}),Q_{\ell}(e_{h,m}-e_{H,m}))_{1,\kappa,\Omega_{\ell}} \\
        &=  2{\kappa}^2 ((I-Q_{\ell})(e_{h,m}-e_{H,m}), Q_{\ell}(e_{h,m}-e_{H,m}))_{L^2(\Omega_{\ell})}  \\
         & \quad\ +\imagunit \kappa ((I-Q_{\ell})(e_{h,m}-e_{H,m}), Q_{\ell}(e_{h,m}-e_{H,m}))_{L^2(\Gamma_{\ell}\cap \Gamma)}  \\ & \quad\ + a_{\ell}((I-Q_{\ell})(e_{h,m}-e_{H,m}), Q_{\ell}(e_{h,m}-e_{H,m})),
\end{aligned}
\end{equation}
where the last term vanishes by the definition \eqref{eqn:defQl} of $a_\ell(\cdot, \cdot)$, then we readily get 
\begin{equation} \label{eqn:QlGalerkinorthloc}
    \sum\limits_{\ell}\abs{((I-Q_{\ell})(e_{h,m}-e_{H,m}),Q_{\ell}(e_{h,m}-e_{H,m}))_{1,\kappa,\Omega_{\ell}}}\leq 2 R_{2,1}(v_h)+R_{2,2}(v_h), 
\end{equation}
where $R_{2,1}$ and $R_{2,2}$ are defined in \eqref{eq:R21}-\eqref{eq:R22} respectively. Using this, we can update 
the upper bound of $R_{\mathrm{loc}}(v_h)$ (with $R_{2}(v_h):= \BowenRevise{2}R_{2,1}(v_h) + R_{2,2}(v_h)$):
\begin{equation}\label{eqn:resloc}
    R_{\mathrm{loc}}(v_h) \BowenRevise{\leq \tilde{C}_{m} \Lambda (1+\kappa C_{st}^{\prime})\beta^{m} } \norm{v_h}_{1,\kappa}^2 + R_{2}(v_h)\,.
\end{equation}
    
To continue the estimation, we now go to bound the energy-norm $\norm{v_h}_{1,\kappa}$ in \eqref{eqn:lowerboundloc} 
from above. We can easily rewrite and bound the norm by the definition of $R_1(v_h)$ in \eqref{eqn:residual1-1loc}:
    \begin{equation}\label{eqn:onesideloc1}
    \begin{aligned}
        \norm{v_h}_{1,\kappa}^{2}  &= \norm{Q_{0,m}v_h}_{1,\kappa}^{2}
        +\norm{e_{h,m}}_{1,\kappa}^{2} + (v_h-e_{h,m},e_{h,m})_{1,\kappa} + (e_{h,m},v_h-e_{h,m} )_{1,\kappa}\\
        & \leq \norm{Q_{0,m}v_h}_{1,\kappa}^{2} + \norm{e_{h,m}}_{1,\kappa}^{2} + 2 R_1(v_h) \\
        & \leq \norm{Q_{0,m}v_h}_{1,\kappa}^{2} + \norm{e_{h,m}-e_{H,m}}_{1,\kappa}^{2} + \norm{e_{H,m}}_{1,\kappa}^{2} + 2R_1(v_h).
    \end{aligned}
    \end{equation}
    On the other hand, we can further bound $R_\mathrm{loc}(v_h)$ from \eqref{eqn:resloc} by using Lemma \ref{lem:res22loc} to get 
    \begin{equation}\label{eqn:reslocbound}
        R_\mathrm{loc}(v_h)\BowenRevise{\leq ( C_{2} \Lambda \kappa H_{\mathrm{sub}} + \tilde{C}_{m} \Lambda (1+\kappa C_{st}^{\prime})\beta^{m} ) }  \norm{v_h}_{1,\kappa}^{2},
    \end{equation}
    then we deduce with \eqref{eqn:reslocbound} and \eqref{eqn:othersidelocfov2} 
    (with \BowenRevise{$ C(\Lambda,\kappa,H_{\mathrm{sub}},\beta)$} from \eqref{eqn:coefR2})
    that 
    {\small
    \begin{equation}\label{eqn:othersidefurthertmp}
        \norm{Q_{0,m}v_h}_{1,\kappa}^{2} + \sum\limits_{\ell}\norm{Q_{\ell}(e_{h,m}-e_{H,m})}_{1,\kappa}^{2} \BowenRevise{\leq \abs{(v_h,Q_{m}^{(1)}v_h)_{1,\kappa}} + R_{1}(v_{h}) + C(\Lambda,\kappa,H_{\mathrm{sub}},\beta)\norm{v_h}_{1,\kappa}^{2}.} 
    \end{equation}}
    \BowenRevise{From the definition of $c_{1}$ we have $\min\{1, (9 C(\Lambda,H,\delta))^{-1} \} \geq c_{1} \Lambda^{-1} $.} 
    Using this and \eqref{eqn:lowbndact}, \eqref{eqn:onesideloc1}, \eqref{eqn:actualE-3}, \BowenRevise{\eqref{eqn:R1m}-\eqref{eqn:SubHm}} and Lemma \ref{lem:res1loc} respectively, we can estimate as follows:
    \begin{equation}\label{eqn:othersidefurtherloc3}
    \begin{aligned}
    &\quad\  \abs{(v_h, Q_{m}^{(1)} v_h)_{1,\kappa}}  \BowenRevise{+ R_{1}(v_{h})} + \BowenRevise{C(\Lambda,\kappa,H_{\mathrm{sub}},\beta)} \norm{v_h}_{1,\kappa}^{2} \BowenRevise{\geq} \norm{Q_{0,m}v_h}_{1,\kappa}^{2} + \sum\limits_{\ell}\norm{Q_{\ell}(e_{h,m}-e_{H,m})}_{1,\kappa,\Omega_{\ell}}^2 \\
    &\geq \min\Big\{1,\frac{1}{\BowenRevise{9} C(\Lambda,H,\delta)}\Big\}(\norm{Q_{0,m}v_h}_{1,\kappa}^{2} + 
    \BowenRevise{9} C(\Lambda,H,\delta)\sum\limits_{\ell}\norm{Q_{\ell}(e_{h,m}-e_{H,m})}_{1,\kappa,\Omega_{\ell}}^2 ) \\
    &\BowenRevise{\geq c_{1} } \Lambda^{-1}(\norm{Q_{0,m}v_h}_{1,\kappa}^{2} + \norm{e_{h,m}-e_{H,m}}_{1,\kappa}^{2}) 
    \BowenRevise{\geq c_{1} } \Lambda^{-1} \Big( \norm{v_h}_{1,\kappa}^{2}- \norm{e_{H,m}}_{1,\kappa}^{2} - \BowenRevise{2} R_1(v_h)\Big) \\
    & \BowenRevise{\geq \frac{1}{2} c_{1} } \Lambda^{-1} \norm{v_h}_{1,\kappa}^{2},
    \end{aligned}
    \end{equation}
    \BowenRevise{which yields for $c_{1}$ that}
    \begin{equation*}
        \abs{(v_h,Q_{m}^{(1)}v_h)_{1,\kappa}} \BowenRevise{+ R_{1}(v_{h})} + \BowenRevise{C(\Lambda,\kappa,H_{\mathrm{sub}},\beta)}\norm{v_h}_{1,\kappa}^{2} 
        \BowenRevise{\geq \frac{1}{2}c_{1}   \Lambda^{-1}\norm{v_{h}}_{1,\kappa}^{2}.}
    \end{equation*}
    Then by simple calculation and recalling \BowenRevise{\eqref{eqn:R1m}-\eqref{eqn:SubHm}, we have} 
    \begin{equation*}
    \begin{aligned}
        \abs{(v_h, Q_{m}^{(1)} v_h)_{1,\kappa}} &\BowenRevise{ \geq  \frac{1}{2}c_{1} \Lambda^{-1}(1- 2c_{1}^{-1} \Lambda ((R_{1}(v_{h}) + C(\Lambda,\kappa,H_{\mathrm{sub}},\beta)) ) )\norm{v_h}_{1,\kappa}^{2}} \BowenRevise{ \geq \frac{1}{4}c_{1} \Lambda^{-1} \norm{v_h}_{1,\kappa}^{2} }.
        \end{aligned}
    \end{equation*} 
    This completes the estimate of the desired lower bound \eqref{eqn:lowerboundloc}. 
\end{proof}

\subsection{Impedance type preconditioner}\label{subsec:Itype} 
The second hybrid Schwarz preconditioner that we propose is an impedance type one, where 
each local solver adopts the impedance boundary condition. More specifically,  
for each subdomain $\Om_\ell$, we define a local solution operator 
$P_{\ell}:H^{1}(\Omega) \to V_{h,\ell}$
such that for each $v\in H^{1}(\Omega)$, we seek $P_{\ell} v\in V_{h,\ell}$ by solving 
\begin{equation}\label{eqn:defPl}
    c_{\ell}(P_{\ell}v, w_{h,\ell}) = a (v, \Pi_h \chi_{\ell}w_{h,\ell}), \ \text{for all}\  w_{h,\ell} \in V_{h,\ell}.
\end{equation}
The well-posedness of $P_{\ell}$ follows from Lemma \ref{lem:ddcontcoer} and the Lax-Milgram theorem. 

Now we are ready to define a second hybrid Schwarz preconditioner: 
\begin{equation}\label{eqn:idealprecond2}
    Q_{m}^{(2)} = Q_{0,m} + (I-Q_{0,m})^{T}\sum\limits_{\ell}\Pi_h \chi_{\ell} P_{\ell}(I-Q_{0,m}).
\end{equation}

We shall analyse the preconditioner $Q_{m}^{(2)}$ under the following additional condition. 
\begin{assum}\label{asm:1.2}
    \begin{minipage}[t]{20cm}
    \vspace{-8pt}
        \begin{AssumpList}[label=(\emph{\arabic*}), ref=(\emph{\arabic*}), align=left] 
        \item[] $\kappa \,\delta \BowenRevise{\geq C_{\mathrm{Isub}}} $, \label{itm:dd2} \\
        \end{AssumpList}
    \end{minipage}
    \BowenRevise{here $C_{\mathrm{Isub}}\geq 1$. If other restriction of $C_{\mathrm{Isub}}$ is required,  we will provide the explicit expression of it accordingly.}
\end{assum}

\subsubsection{A priori estimates in the subdomains}\label{subsubsec:Qm2}
We first present several a priori estimates that are needed later for the analysis of the 
preconditioner $Q_{m}^{(2)}$. 
The first lemma provides a reverse stable splitting of the partition of unity 
\cite[Lemma 3.1]{graham_domain_2020}.
\begin{lemma} \label{lem:reversedd} 
    The following estimate is valid for all $v \in H^{1}(\Omega)$: 
    \begin{equation}\label{eq:Clambda}
        \sum\limits_{\ell}^{}\norm{\chi_{\ell}v}_{1,\kappa,\Omega_{\ell}}^{2} \geq C(\Lambda, \kappa,\delta) \norm{v}_{1,\kappa}^{2},
    \end{equation}
    where 
    $C(\Lambda,\kappa,\delta)=\Lambda^{-1}- \BowenRevise{2 C_{p}} \Lambda/(\kappa\delta)$.
\end{lemma}

\begin{lemma}\label{lem:upbndPl}
    Under Assumptions\,\ref{asm:1.2} and \ref{asm:2}\ref{itm:polutefree}, it holds \BowenRevise{for $C_{P,\ell}= 3 + \frac{4 C_{p}}{\kappa\delta_{\ell}} + 6C_{\Pi} + 4C_{c} C_{\Pi} $,}
    \begin{equation} \label{eqn:upbndPl}
        \norm{P_{\ell}v_h}_{1,\kappa,\Omega_{\ell}} \BowenRevise{\leq C_{P,\ell}} \norm{v_h}_{1,\kappa,\Omega_{\ell}}\,, \q {\rm for \,\, all } \,\,v_h \in V_h\,.
    \end{equation}
\end{lemma}
\begin{proof}
    For the desired bound, we need an estimate of the following cross-term: 
    \begin{equation}\label{eqn:bldef}
    \begin{aligned}
        b_{\ell}(v,w) &:= (v,\chi_{\ell}w)_{1,\kappa,\Omega_{\ell}} - (\chi_{\ell}v,w)_{1,\kappa,\Omega_{\ell}} 
         = c_{\ell}(v,\chi_{\ell}w) - c_{\ell}(\chi_{\ell}v,w) \\
        & = \int_{\Omega_{\ell}}\nabla\chi_{\ell}\cdot (\bar{w}\nabla v- v \nabla \bar{w})\,, \q 
        \text{for all } \, v, w \in H^{1}(\Omega_{\ell})\,,
    \end{aligned} 
    \end{equation}
     for which we use the property of $\chi_{\ell}$ and the Cauchy-Schwarz inequality 
     to readily obtain
    \begin{equation} \label{eqn:bl}
        \abs{b_{\ell}(v,w)} \BowenRevise{\leq} \frac{\BowenRevise{C_{p}}}{\kappa\delta_{\ell}}\norm{v}_{1,\kappa,\Omega_{\ell}}\norm{w}_{1,\kappa,\Omega_{\ell}}.
    \end{equation}
    
    We are now ready to bound $P_{\ell}v_h$. By writing $P_{\ell}v_h=\Pi_h\chi_{\ell} v_h + v_{h,\ell}$, with 
    $v_{h,\ell}:= P_{\ell}v_h - \Pi_h\chi_{\ell} v_h$. 
    The first term can be bounded directly by using \eqref{eqn:energynormchi} and \eqref{eqn:nodint} \BowenRevise{and recalling $\kappa h$,$\frac{h}{\delta} \leq 1$,} to get 
    \begin{equation} \label{eqn:Pihpoulocstable}
    \norm{\Pi_h\chi_{\ell}v_h}_{1,\kappa,\Omega_{\ell}} \leq 
    \norm{(I-\Pi_h)\chi_{\ell}v_h}_{1,\kappa,\Omega_{\ell}} + \norm{\chi_{\ell}v_h}_{1,\kappa,\Omega_{\ell}} \BowenRevise{\leq} \BowenRevise{ 
    (1 + \frac{\BowenRevise{C_{p}}}{\kappa\delta_{\ell}} + 2 C_{\Pi} )} \norm{v_h}_{1,\kappa,\Omega_{\ell}}.
    \end{equation}
To estimate $v_{h,\ell} = P_{\ell}v_h - \Pi_h\chi_{\ell} v_h$, we obtain from 
\eqref{eqn:coercl} and the definition of $P_{\ell}$ that 
    \begin{equation}\label{eqn:vhl1}
        \norm{v_{h,\ell}}_{1,\kappa,\Omega_{\ell}}^2 =  \Re c_{\ell}(v_{h,\ell},v_{h,\ell}) = \Re(a(v_{h},\Pi_h\chi_{\ell}v_{h,\ell}) - c_{\ell}(\Pi_h\chi_{\ell}v_{h},v_{h,\ell})).
    \end{equation}
    By noting that $\chi_{\ell}$ vanishes at $\partial\Omega_{\ell}\backslash\Gamma$ and 
    the definition of $b_{\ell}(\cdot, \cdot)$ and \eqref{eqn:Pihpoulocstable}, we can rewrite 
    \begin{equation}\label{eqn:vhl2}
    \begin{aligned}
    &a(v_{h},\Pi_h\chi_{\ell}v_{h,\ell}) - c_{\ell}(\Pi_h\chi_{\ell}v_{h},v_{h,\ell}) \\
    = & c_{\ell}(v_h,\Pi_h\chi_{\ell}v_{h,\ell}) - c_{\ell}(\Pi_h\chi_{\ell}v_h,v_{h,\ell}) - 2\kappa^2(v_h,\Pi_h\chi_{\ell}v_{h,\ell}) \\
    = &b_{\ell}(v_h,v_{h,\ell}) - 2\kappa^2(v_h,\Pi_h\chi_{\ell}v_{h,\ell}) + c_{\ell}(v_h,(\Pi_h-I)\chi_{\ell}v_{h,\ell}) - c_{\ell}((\Pi_h-I)\chi_{\ell}v_h,v_{h,\ell}).
    \end{aligned}
    \end{equation}
    Combining \eqref{eqn:vhl1} and \eqref{eqn:vhl2}, along with the 
 continuity of $c_{\ell}(\cdot, \cdot)$ in Lemma \ref{lem:ddcontcoer}, the interpolation in Lemma \ref{lem:nodint} and 
 the bound of $b_{\ell}(\cdot, \cdot)$ in \eqref{eqn:bl}, we come to 
\begin{equation}\label{eqn:vhl}
    \norm{v_{h,\ell}}_{1,\kappa,\Omega_{\ell}} \BowenRevise{\leq} \BowenRevise{( 2 + \frac{3 C_{p}}{\kappa \delta_{\ell}} + 4 C_{\Pi} +  4 C_{c}C_{\Pi}  )  }\norm{v_h}_{1,\kappa,\Omega_{\ell}}.
\end{equation}
Now recalling $P_{\ell}v_h=\Pi_h\chi_{\ell} v_h + v_{h,\ell}$ and applying the triangle inequality, 
we readily derive 
\begin{equation*}
    \norm{P_{\ell}v_h}_{1,\kappa,\Omega_{\ell}}\BowenRevise{\leq} 
    \norm{v_{h,\ell}}_{1,\kappa,\Omega_{\ell}} +\norm{\Pi_h\chi_{\ell}v_h}_{1,\kappa,\Omega_{\ell}} \BowenRevise{\leq ( 3 + \frac{4 C_{p}}{\kappa\delta_{\ell}} + 6C_{\Pi} + 4C_{c} C_{\Pi} )} \norm{v_h}_{1,\kappa,\Omega_{\ell}}.
\end{equation*}
\end{proof}

Recalling the cross-term $b_\ell(\cdot, \cdot)$ in \eqref{eqn:bldef}, 
we now bound the following two terms that are needed to finalize the optimality estimate 
of the preconditioner $Q_m^{(2)}$: 
\begin{align}
    R_{3,1}(v_h)&:=  \sum\limits_{\ell} \abs{b_{\ell}(e_{h,m},P_{\ell}e_{h,m})}, \label{eq:R31}\\
    R_{3,2}(v_h)&:= \sum\limits_{\ell}\norm{\chi_{\ell}e_{h,m}}_{1,\kappa,\Omega_{\ell}}\norm{z_{h,\ell}}_{1,\kappa,\Omega_{\ell}} \q \text{(with } \, z_{h,\ell}:= P_{\ell}e_{h,m} - \Pi_h\chi_{\ell} e_{h,m})\,. \label{eq:R32} 
\end{align}

\begin{lemma} \label{lem:res32loc}
Under Assumptions \ref{asm:1.1}, \ref{asm:2} and \eqref{eqn:resolcond}, the following estimates are valid 
for all $v_h \in V_h$: 
    \begin{align}
        R_{3,1}(v_h) \BowenRevise{\leq \frac{C_{3,1}}{\kappa\delta} }  \norm{v_h}_{1,\kappa}^{2}, \quad R_{3,2}(v_h) \BowenRevise{\leq (C_{3,2}^{(1)}C_{3,2}^{(2)})^{\frac{1}{2}}} \norm{v_h}_{1,\kappa}^{2}, \label{eqn:res3loc}
    \end{align}
    \BowenRevise{where $C_{P} = \max_{\ell}C_{P,\ell}$, $C_{3,1} = \Lambda C_{p} C_{P}(C_{m}^{\prime})^{2}$, $C_{3,2}^{(1)}=\Lambda(1+{C_{p}})(C_{m}^{\prime})^{2}$ and 
\begin{equation*}
    C_{3,2}^{(2)}\! =\!  2\Lambda(\frac{ ((C_{p} \!+\! 4C_{c}C_{\Pi})^{2} \!+\! 4 C_{p}^{2} ) (C_{m}^{\prime})^{2}}{\kappa^{2} \delta^{2}} + 8 C_{m}^{2}(1 \!+\! 2C_{\Pi} )^{2}( (C_{I}n_o\kappa H)^{2} \!+\! 4 (C_{exp}^{\prime\prime})^{2} (1 \!+\! \kappa C_{st}^{\prime})^{2}\beta^{2m}) ).
\end{equation*}}  
\end{lemma}
\vskip-0.5truecm
\BowenRevise{
For simplicity, we use $ R_{3,2}(v_h)\leq C_{3,2}^{\RN{1}}({\kappa\delta})^{-1}+C_{3,2}^{\RN{2}}\kappa H +C_{3,2}^{\RN{3}}(1+\kappa C_{st}^{\prime})\beta^{m}$ for the  analysis.
}
\begin{proof}
 We first bound $R_{3,1}(v_h) $. For every $\ell$, we get from the estimate of \eqref{eqn:bl} that 
   \begin{equation*}
           \abs{b_{\ell}(e_{h,m},P_{\ell}e_{h,m})}
           \BowenRevise{\leq} \frac{\BowenRevise{C_{p}}}{\kappa\delta_{\ell}}\norm{e_{h,m}}_{1,\kappa,\Omega_{\ell}}\norm{P_{\ell}e_{h,m}}_{1,\kappa,\Omega_{\ell}}.
   \end{equation*}
   Summing over all $\ell$'s and then using \eqref{eqn:upbndPl}, 
\eqref{eqn:actualE-1} and \eqref{eqn:finiteoverlap}, we deduce 
\begin{equation*}
\begin{aligned}
    R_{3,1}(v_h) & \BowenRevise{\leq} \sum\limits_{\ell}\frac{\BowenRevise{C_{p}}}{\kappa\delta_{\ell}}\norm{e_{h,m}}_{1,\kappa,\Omega_{\ell}}\norm{P_{\ell}e_{h,m}}_{1,\kappa,\Omega_{\ell}}
    \BowenRevise{\leq } \sum\limits_{\ell}\frac{\BowenRevise{C_{p} C_{P,\ell}}}{\kappa\delta_{\ell}}\norm{e_{h,m}}_{1,\kappa,\Omega_{\ell}}^{2} \BowenRevise{\leq  \frac{\Lambda C_{p}C_{P}(C_{m}^{\prime})^{2}}{\kappa\delta}  } \norm{v_h}_{1,\kappa}^{2}.
\end{aligned}
\end{equation*}

Next, we bound $R_{3,2}(v_h)$. 
By the definitions of $P_{\ell}$ and $z_{h,\ell}$, we have 
\begin{equation} \label{eqn:commuteeqnloc}
    c_{\ell}(z_{h,\ell}, z_{h,\ell}) =c_{\ell}(P_{\ell}e_{h,m} - \Pi_h\chi_{\ell}e_{h,m} , z_{h,\ell})=a(e_{h,m}, \Pi_h\chi_{\ell}z_{h,\ell}) - c_{\ell}(\Pi_h\chi_{\ell}e_{h,m}, z_{h,\ell}).
\end{equation}
The same as in \eqref{eqn:vhl2}, we first rewrite the right-hand side in \eqref{eqn:commuteeqnloc}, then 
apply the coercivity and continuity of $c_{\ell}$ in \eqref{eqn:coercl}-\eqref{eqn:contddc}, 
the estimates \eqref{eqn:bl}-\eqref{eqn:Pihpoulocstable} and \eqref{eqn:nodint} respectively to obtain 
\begin{equation*} \label{eqn:estzhlloc}
    \begin{aligned}
    &\quad \ \norm{z_{h,\ell}}_{1,\kappa,\Omega_{\ell}}^{2}  =\!  \Re c_{\ell}(z_{h,\ell},z_{h,\ell})\\ 
    & =\! \Re (b_{\ell}(e_{h,m}, z_{h,\ell}) - 2 \kappa^2 (e_{h,m} ,\Pi_h\chi_{\ell}z_{h,\ell} )_{\Omega_{\ell}} + c_{\ell}(e_{h,m},(\Pi_h-I)\chi_{\ell}z_{h,\ell}) - c_{\ell}((\Pi_h-I)\chi_{\ell}e_{h,m},z_{h,\ell}))\\
    & \!\BowenRevise{\leq}  \BowenRevise{  (\frac{C_{p}}{\kappa\delta_{\ell}} + 4C_{c}C_{\Pi}\frac{h_{\ell}}{\delta_{\ell}} ) } \norm{e_{h,m}}_{1,\kappa,\Omega_{\ell}}\norm{z_{h,\ell}}_{1,\kappa,\Omega_{\ell}} + \BowenRevise{2 \kappa (1 + \frac{\BowenRevise{C_{p}}}{\kappa\delta_{\ell}} + 2 C_{\Pi} )} \norm{e_{h,m}}_{L^2(\Omega_{\ell})}\norm{z_{h,\ell}}_{1,\kappa,\Omega_{\ell}},
    \end{aligned}
\end{equation*}
then a cancellation of the common factor $\norm{z_{h,\ell}}_{1,\kappa,\Omega_{\ell}}$ on both sides implies immediately 
\begin{equation}
    \norm{z_{h,\ell}}_{1,\kappa,\Omega_{\ell}} \BowenRevise{\leq} \BowenRevise{\frac{C_{p}+4C_{c}C_{\Pi}}{\kappa\delta_{\ell}}}\norm{e_{h,m}}_{1,\kappa,\Omega_{\ell}} + \BowenRevise{2 \kappa (1 + \frac{\BowenRevise{C_{p}}}{\kappa\delta_{\ell}} + 2 C_{\Pi} )} \norm{e_{h,m}}_{L^2(\Omega_{\ell})}.
\end{equation}
Summing both sides over all subdomains, along with \eqref{eqn:finiteoverlap}, \eqref{eqn:actualL2-1}, \eqref{eqn:actualE-1} and \eqref{eqn:resolcond}, \BowenRevise{and noticing $\kappa\delta>1$}, we arrive at 
\begin{equation}\label{eqn:res3cauchyloc}
    \begin{aligned}
       \sum\limits_{\ell}\norm{z_{h,\ell}}_{1,\kappa,\Omega_{\ell}}^{2} & \BowenRevise{\leq} \sum\limits_{\ell}\frac{\BowenRevise{2(C_{p}+4C_{c}C_{\Pi})^{2}}}{\kappa^2\delta_{\ell}^2}\norm{e_{h,m}}_{1,\kappa,\Omega_{\ell}}^{2} + \sum\limits_{\ell}\BowenRevise{8(\kappa^{2}(1+2C_{\Pi})^{2} + \frac{C_{p}^{2}}{\delta_{\ell}^{2}} ) } \norm{e_{h,m}}_{L^2(\Omega_{\ell})}^{2}\\  & \BowenRevise{\leq 2\Lambda(\frac{ ((C_{p}+4C_{c}C_{\Pi})^{2}+ 4 C_{p}^{2} ) (C_{m}^{\prime})^{2}}{\kappa^{2} \delta^{2}} +} \\
       & \quad\quad \BowenRevise{8 C_{m}^{2}(1 \!+\! 2C_{\Pi} )^{2}( (C_{I}n_o\kappa H)^{2} \!+\! 4 (C_{exp}^{\prime\prime})^{2} (1 \!+\! \kappa C_{st}^{\prime})^{2}\beta^{2m}) ) }
       \norm{v_h}_{1,\kappa}^{2}.
    \end{aligned}
\end{equation}
Now \eqref{eqn:res3loc} follows from Cauchy-Schwarz inequality, 
\eqref{eqn:actualE-1}, \eqref{eqn:res3cauchyloc}, \eqref{eqn:energynormchi} and Assumption\,\ref{asm:2}. 
\end{proof}

\subsubsection{Optimality estimates of the second preconditioner $Q_{m}^{(2)}$}\label{subsubsec:optimal_Qm2}
We are now ready to demonstrate the optimality of the preconditioner $Q_{m}^{(2)}$, namely, 
we establish the upper bound of the energy-norm and the lower bound 
of the field of values of $Q_{m}^{(2)}$, both of which are proved to be independent of the key parameters, 
$k$, $h$, $H$, $\delta$ and $H_{\mathrm{sub}}$. 

\begin{theorem}\label{thm2:upboundact}
    Under Assumptions \ref{asm:1.1}, \ref{asm:2} and \eqref{eqn:resolcond}, we have the upper bound 
    of $Q_{m}^{(2)}$: 
    \begin{equation}
        \norm{Q_{m}^{(2)} v_h}_{1,\kappa} \BowenRevise{\leq (1+ C_{m}^{\prime} +  \Lambda C_{P} C_{m}^{\prime} (1 + \frac{\BowenRevise{C_{p}}}{\kappa\delta_{\ell}} + 2 C_{\Pi} )  )}  \norm{v_h}_{1,\kappa}\,, 
        \q {\rm for \,\,all } \,\, v_h\in V_h\,.
    \end{equation}
\end{theorem}
\begin{proof}
For every $\ell$, we notice that $\Pi_h\chi_{\ell}e_{h,m} \in H^1(\Omega_{\ell})$ and 
vanishes at $\Gamma_{\ell}\backslash \Gamma$. Then by a finite overlap argument \eqref{eqn:finiteoverlap}, and using \eqref{eqn:upbndPl}, \eqref{eqn:Pihpoulocstable} and Lemma \ref{lem:actualL2} we get 
\begin{align*}
            \norm{Q_{m}^{(2)}v_h}_{1,\kappa} & \leq  \norm{Q_{0,m}v_h}_{1,\kappa} + \norm{(I-Q_{0,m}^{T})\sum\limits_{\ell}\Pi_h\chi_{\ell}P_{\ell}e_{h,m}}_{1,\kappa} \nonumber \\
            & \BowenRevise{\leq} \BowenRevise{ (1 + C_{m}^{\prime}) }\norm{v_h}_{1,\kappa} + \BowenRevise{ C_{m}^{\prime} }\sum\limits_{\ell}\norm{\Pi_{h}\chi_{\ell}P_{\ell}e_{h,m}}_{1,\kappa,\Omega_{\ell}} \\
            & \BowenRevise{\leq} \BowenRevise{ (1 + C_{m}^{\prime}) } \norm{v_h}_{1,\kappa} + \BowenRevise{C_{m}^{\prime}}\sum\limits_{\ell}\BowenRevise{(1 + \frac{\BowenRevise{C_{p}}}{\kappa\delta_{\ell}} + 2 C_{\Pi} )}\norm{P_{\ell}e_{h,m}}_{1,\kappa,\Omega_{\ell}} \\  
            &\BowenRevise{\leq (1+ C_{m}^{\prime} +  \Lambda C_{P} C_{m}^{\prime} (1 + \frac{\BowenRevise{C_{p}}}{\kappa\delta_{\ell}} + 2 C_{\Pi} )  )} \norm{v_h}_{1,\kappa}.
\end{align*}
\end{proof}

\BowenRevise{Similar to Theorem \ref{thm:lowerboundact}, we first give specific Assumptions \ref{asm:1.1}, \ref{asm:2} and \eqref{eqn:resolcond} for Theorem \ref{thm2:lowerboundact}. For the constant $C(\Lambda,\kappa,\delta)$ in \eqref{eq:Clambda}, when
\begin{equation} \label{eqn:overlapQ2}
    \kappa \delta \geq 2\Lambda^{2}C_{p}(1-\tilde c_0)^{-1},
\end{equation} 
we have $C(\Lambda,\kappa,\delta)  \geq \tilde c_0{\Lambda}^{-1}$, here $0<\tilde c_0<1$ is a generic constant. When 
\begin{equation}\label{eqn:R1mQ2}
    m \geq \abs{\log \beta}^{-1} \log \big((1+\kappa C_{st}^{\prime}) \Theta_{1}((2\tilde c_0)^{-1}\Lambda)\big),\quad  \kappa H \leq \Theta_{2}((2\tilde c_0)^{-1}\Lambda),
\end{equation}
we have $R_{1}(v_{h}) \leq \tilde c_0({4\Lambda})^{-1}\norm{v_{h}}_{1,\kappa}^{2}$. Let 
$ C( \Lambda,\kappa,\delta,H,\beta):=C_3({\kappa\delta})^{-1}+C_{3,2}^{\RN{2}}\kappa H +C_{3,2}^{\RN{3}}(1+\kappa C_{st}^{\prime})\beta^{m}$, where $C_3=C_{3,1}+C_{3,2}^{\RN{1}}+ 2 \Lambda C_{\Pi} (1 + C_{P}+ C_{p}) (C_{m}^{\prime})^{2} $, when 
\begin{equation}\label{eqn:CQ2}
    m \geq \abs{\log \beta}^{-1} \log \big(12\Lambda C_{3,2}^{\RN{3}}(\tilde c_0)^{-1} (1+\kappa C_{st}^{\prime})\big), \:  \kappa H \leq (12\Lambda(\tilde c_0)^{-1} C_{3,2}^{\RN{2}})^{-1}, \: \kappa \delta > 12C_3\Lambda(\tilde c_0)^{-1},
\end{equation}
we have $C( \Lambda,\kappa,\delta,H,\beta)\leq \tilde c_0(4\Lambda)^{-1}$. Next we use \eqref{eqn:overlapQ2}-\eqref{eqn:CQ2}  as the specific assumption for Theorem \ref{thm2:lowerboundact}.
}

\begin{theorem}\label{thm2:lowerboundact}
    Under Assumptions \ref{asm:1.2}, \ref{asm:2} and \eqref{eqn:resolcond}, 
    it holds that 
    \begin{equation} \label{eqn:lowerboundloc2}
        \abs{( Q_{m}^{(2)}v_h,v_h )_{1,\kappa}} \BowenRevise{\geq  \frac{\tilde c_0}{2\Lambda}}\norm{v_h}_{1,\kappa}^{\BowenRevise{2}}\,, \q {\rm for \,\,all } \,\, v_h\in V_h\,.
    \end{equation}
\end{theorem}
\begin{proof}
Noting that $\chi_{\ell}$ vanishes at $\partial\Omega_{\ell}\backslash \Gamma$, 
by the definitions of $Q_{m}^{(2)}$ and $b_{\ell}$ we can rewrite 
\begin{equation*} \label{eqn:lowerbndexploc2}
\begin{aligned}
    &\quad \ (v_h, Q_{m}^{(2)}v_h)_{1,\kappa}\\
    & =\! (v_h, Q_{0,m}v_h)_{1,\kappa} + \sum\limits_{\ell}(e_{h,m}, \Pi_h\chi_{\ell}P_{\ell}e_{h,m})_{1,\kappa,\Omega_{\ell}} \\
    & =\! (v_h, Q_{0,m}v_h)_{1,\kappa}\! +\! \sum\limits_{\ell}(e_{h,m},\chi_{\ell}P_{\ell}e_{h,m})_{1,\kappa,\Omega_{\ell}} + (e_{h,m}, (\Pi_h-I)\chi_{\ell}P_{\ell}e_{h,m})_{1,\kappa,\Omega_{\ell}} \\
    &=\! (v_h, Q_{0,m}v_h)_{1,\kappa}\! + \!\sum\limits_{\ell}((\chi_{\ell}e_{h,m},P_{\ell}e_{h,m})_{1,\kappa,\Omega_{\ell}} + b_{\ell}(e_{h,m},P_{\ell}e_{h,m}) + (e_{h,m}, (\Pi_h-I)\chi_{\ell}P_{\ell}e_{h,m})_{1,\kappa,\Omega_{\ell}} \\
    & =\! \norm{Q_{0,m}v_h}_{1,\kappa}^{2} + (e_{h,m},v_h-e_{h,m})_{1,\kappa} + \sum\limits_{\ell}\norm{\chi_{\ell}e_{h,m}}_{1,\kappa,\Omega_{\ell}}^{2} \\
    & \quad -\sum\limits_{\ell}((\chi_{\ell}e_{h,m},\chi_{\ell}e_{h,m}-P_{\ell}e_{h,m})_{1,\kappa,\Omega_{\ell}} -b_{\ell}(e_{h,m},P_{\ell}e_{h,m}) - (e_{h,m}, (\Pi_h-I)\chi_{\ell}P_{\ell}e_{h,m})_{1,\kappa,\Omega_{\ell}} .
    \end{aligned}
\end{equation*}
We then estimate the last and fourth terms above. 
By the error estimate \eqref{eqn:nodint}, we easily get  
\begin{equation*}
    \abs{ (e_{h,m}, (\Pi_h-I)\chi_{\ell}P_{\ell}e_{h,m})_{1,\kappa,\Omega_{\ell}} } 
    \BowenRevise{\leq 2 C_{\Pi} C_{P,\ell}  \frac{h_{\ell}}{\delta_{\ell}} } \norm{e_{h,m}}_{1,\kappa,\Omega_{\ell}}^{2},
\end{equation*}
while by the definitions of $R_{3,2}(v_h)$ and $z_{h,\ell}$ (cf.\,\eqref{eq:R32}) and the estimate \eqref{eqn:nodint} 
and \eqref{eqn:energynormchi}, we derive 
\begin{equation*}
    \begin{aligned}
        &\sum\limits_{\ell} \abs { (\chi_{\ell}e_{h,m},\chi_{\ell}e_{h,m}-P_{\ell}e_{h,m})_{1,\kappa,\Omega_{\ell}} } \\
         \BowenRevise{\leq}  & \sum\limits_{\ell} (\abs{(\chi_{\ell}e_{h,m},\chi_{\ell}e_{h,m}-\Pi_h\chi_{\ell}e_{h,m})_{1,\kappa,\Omega_{\ell}}} + \abs{ (\chi_{\ell}e_{h,m},z_{h,\ell})_{1,\kappa,\Omega_{\ell}} } )\\
         \BowenRevise{\leq} & \BowenRevise{\sum\limits_{\ell} 2 C_{\Pi} (1+\frac{C_{p}}{\kappa\delta_{\ell}})\frac{h_{\ell}}{\delta_{\ell}} } 
         \norm{e_{h,m}}_{1,\kappa,\Omega_{\ell}}^{2}+ R_{3,2}(v_h).
    \end{aligned}
\end{equation*}
Using the above two estimates, we readily deduce the following bound of $(v_h, Q_{m}^{(2)}v_h)_{1,\kappa}$:
\begin{equation}\label{eq:Qm2}
    \begin{aligned}
        \abs{(v_h, Q_{m}^{(2)}v_h)_{1,\kappa} }\BowenRevise{\geq} \norm{Q_{0,m}v_h}_{1,\kappa}^{2} + \sum\limits_{\ell}\norm{\chi_{\ell}e_{h,m}}_{1,\kappa,\Omega_{\ell}}^{2} - R_1(v_h) - R_3(v_h), 
    \end{aligned}
\end{equation}
where $R_1(v_h)$ is defined in \eqref{eqn:residual1-1loc} and $R_3(v_h)$ is given by
\begin{equation*}
    R_{3}(v_h):= R_{3,1}(v_h) + R_{3,2}(v_h) + \BowenRevise{2 \Lambda C_{\Pi} (1 + C_{P}+ \frac{C_{p}}{\kappa\delta}) (C_{m}^{\prime})^{2} \frac{1}{\kappa\delta}  } \norm{v_{h}}_{1,\kappa}^{2}.
\end{equation*}
By using Lemma \ref{lem:reversedd}, we obtain immediately from \eqref{eq:Qm2} that 
\begin{equation*}
    \abs{(v_h, Q_{m}^{(2)}v_h)_{1,\kappa} } + R_1(v_h) + R_3(v_h) \BowenRevise{\geq} \norm{Q_{0,m}v_h}_{1,\kappa}^{2} + C(\Lambda,\kappa,\delta)\norm{e_{h,m}}_{1,\kappa}^{2},
\end{equation*}
\BowenRevise{
which, combining with \eqref{eqn:overlapQ2} and triangle inequality}, implies
\begin{equation*}
    \begin{aligned}
        \abs{(v_h, Q_{m}^{(2)}v_h)_{1,\kappa} } + R_1(v_h) + R_3(v_h)  \geq \BowenRevise{\frac{\tilde c_0}{\Lambda}} (\norm{Q_{0,m}v_h}_{1,\kappa}^{2} + \norm{e_{h,m}}_{1,\kappa}^{2}) \geq  \BowenRevise{\frac{\tilde c_0}{\Lambda}}\norm{v_h}_{1,\kappa}^{2}.
    \end{aligned}
\end{equation*} 
To proceed further, we apply \eqref{eqn:overlapQ2}-\eqref{eqn:CQ2} to get 
\begin{equation*} \label{eqn:finallocQ2}
    \abs{(v_h, Q_{m}^{(2)}v_h)_{1,\kappa} } \geq (\BowenRevise{\frac{\tilde c_0}{\Lambda}}-C( \Lambda,\kappa,\delta,H,\beta) ) \norm{v_{h}}_{1,\kappa}^{2}\BowenRevise{-R_1(v_h)\geq  \frac{\tilde c_0}{2\Lambda} \norm{v_{h}}_{1,\kappa}^{2},}
\end{equation*}
\BowenRevise{which is the desired estimate.}
\end{proof}
\BowenRevise{\begin{remark}
The GMRES convergence theory in \cite{Elman83} combines Theorem \ref{thm:upboundact}-\ref{thm:lowerboundact} and Theorem \ref{thm2:upboundact}-\ref{thm2:lowerboundact} indicate that if GMRES is applied to the finite element system \eqref{eqn:Galerkinform}
in the $(\cdot,\cdot)_{1,\kappa}$ inner product with preconditioners $ Q_{m}^{(1)}$ and $ Q_{m}^{(2)}$, then the number of iterations needed to achieve a prescribed accuracy
remains bounded.
\end{remark} }
\section{Numerical Experiments}\label{sec:numexp}
In this section we present various numerical experiments to check the performance of the GMRES iteration 
applied for solving the finite element system \eqref{eqn:Galerkinform} 
with two new preconditioners $Q_m^{(1)}$ and $Q_m^{(2)}$ proposed in Sections \ref{subsec:Dtype} and \ref{subsec:Itype}. 
We will consider different choices of the key parameters, such as $H$, $H_{\mathrm{sub}}$, $\delta$ and $m$, 
to find out how these parameters may affect the performance of GMRES. 
We will also demonstrate some examples that are not fully covered by the theory in this work.

We consider the computational domain $\Om$ to be the unit square, i.e., $\Omega = (0,1)^2$, 
and a uniform coarse triangulation $\mathcal{T}^{H}$ of diameter $H$ on $\Om$, which is further refined 
uniformly to get a fine triangulation $\mathcal{T}^{h}$ of diameter $h$. We use the linear finite element spaces 
$V_H$ and $V_h$ associated with $\mathcal{T}^{H}$ and $\mathcal{T}^{h}$ respectively. 
$\{\phi_{\ell} \}$ is a set of nodal basis functions with vertices being the nodes of 
$\mathcal{T}^{H}$ and a support of diameter $H_{\ell}$. We choose the partition of unity $\chi_{\ell}=\phi_{\ell}$,
the subdomains $\Omega_{\ell} = \mathrm{supp}\ \phi_{\ell}$ (if not stated otherwise), 
with the subdomain size $H_{\ell}=H_{\mathrm{sub}}$ for all $\ell$, where $H_{\mathrm{sub}}$ is some multiple 
of $H$. 
We have made the experiments with
$f$ and $g$  obtained by the planar wave propagation solution $u = e^{\imagunit \kappa \mathbf{d}\cdot \mathbf{x}}$ 
with $\mathbf{d}=(1/\sqrt{2},1/\sqrt{2})$.

In all the experiments we choose the fine mesh size $h \sim  \kappa^{-3/2}$ 
(roughly the mesh that is believed to keep the relative error of the finite element 
solution bounded independently of $\kappa$). We consider only the 2D problems here, 
as the DOFs of the systems would grow like $(\kappa^{3/2})^3 = \kappa^{4.5}$ in 3D, very rapidly with $\kappa$, 
and the system would become too large-scale when we check the numerical performance 
with a reasonable high wave number, such as $\kappa = 500$. 
In all the examples we set all the key parameters involved in two new preconditioners to be closest to those 
predicted scales as in the convergence theory, 
i.e., choose coarse mesh size $H$, subdomain size $H_{\mathrm{sub}}$, overlap size $\delta$, 
oversampling parameter $m$ as functions of $\kappa$ such that they meet 
Assumption \ref{asm:1.1} for the first preconditioner $Q_m^{(1)}$ and Assumption \ref{asm:1.2} 
for the second preconditioner $Q_m^{(2)}$. 


We solve the finite element system \eqref{eqn:Galerkinform} 
by the standard GMRES (with residual minimization in the Euclidean norm).
GMRES  is terminated when the initial residual is reduced by $10^{-6}$, with  
all initial guesses set to zero. 
All numerical experiments were implemented with Matlab R2023a and on an Intel Xeon(R) Gold 6248R CPU.






\paragraph{{\bf Example 1}} \label{para:exp1}
In this example, we demonstrate the dependence of GMRES iterations on the growth of 
$\kappa$. 
For $Q^{(1)}_{m}$, we choose $h\sim \kappa^{-\frac{3}{2}}$, $H=\delta=H_{\mathrm{sub}}/2  \sim \kappa^{-1}$, $m=\log_2(\kappa)-1$, 
while for $Q^{(2)}_{m}$, we select
$h\sim \kappa^{-\frac{3}{2}}$, $H\sim \kappa^{-1}$,  $H_{\mathrm{sub}} = 4H$, $\delta=2H$, 
$m=\log_2(\kappa)-1$.

In these cases we expect the number of iterations is not affected by varying $\kappa $ and mesh size $h$
based on our convergence theory. 
This is clearly confirmed by the data in Tables \ref{tab:table1-1} and \ref{tab:table1-2}.
\vskip-0.35truecm
\begin{table}[H]
\centering
\parbox{.45\linewidth}{
    \centering
    \caption{GMRES iterations of $Q^{(1)}_{m}$}
    \vskip-0.2truecm
    \begin{tabular}{cc}
    \hline
    $\kappa$ & iter \\ \hline
    16  & 9    \\
    32  & 8    \\
    64  & 8    \\
    128 & 8    \\
    256 & 8    \\ 
    500 & 8    \\\hline
    \end{tabular}
    \label{tab:table1-1}
}
\parbox{.45\linewidth}{
\centering
\caption{GMRES iterations of $Q^{(2)}_{m}$}
\vskip-0.2truecm
\begin{tabular}{cc}
\hline
$\kappa$ & iter \\ \hline
16  & 7    \\
32  & 7    \\
64  & 7    \\
128 & 7    \\
256 & 7    \\ 
500 & 7    \\\hline
\end{tabular}
\label{tab:table1-2}
}
\end{table}
We next show the $h$-independence of the new preconditioners. In these cases we fix $\kappa = 80$, 
$H\sim \kappa^{-1}$ and $m=2$. We choose
$H=\delta=H_{\mathrm{sub}}/2$ for $Q_{m}^{(1)}$, and 
$\delta=2H$ and $H_{\mathrm{sub}}=4H$  for $Q_{m}^{(2)}$. 
From a prior error estimates of the Helmholtz equation \cite{zhu_preasymptotic_2013}, we know that as $h$ becomes smaller, 
the finite element solution $u_h$ will be more accurate. 
Tables \ref{tab:table1-3} and \ref{tab:table1-4} show a clear $h$-independence of the new preconditioners 
even though the coarse solution is still not so accurate. 
\vskip-0.35truecm
\begin{table}[H]
\centering
\parbox{.45\linewidth}{
    \centering
    \caption{GMRES iterations of $Q^{(1)}_{m}$}
    \vskip-0.2truecm
    \begin{tabular}{cc}
    \hline
    $h$ & iter \\ \hline
    $2^{-10}$  & 10    \\
    $2^{-11}$ & 9    \\
    $2^{-12}$  & 9    \\
    $2^{-13}$ & 9    \\ \hline
    \end{tabular}
    \label{tab:table1-3}
}
\parbox{.45\linewidth}{
\centering
\caption{GMRES iterations of $Q^{(2)}_{m}$}
\vskip-0.2truecm
    \begin{tabular}{cc}
    \hline
    $h$ & iter \\ \hline
    $2^{-10}$   & 9   \\
    $2^{-11}$  & 8    \\
    $2^{-12}$  & 8    \\
    $2^{-13}$ & 8    \\ \hline
    \end{tabular}
\label{tab:table1-4}
}
\end{table}

\paragraph{{\bf Example 2}}\label{para:exp2} 

In this example we show the dependence of GMRES iterations on the oversampling parameter $m$.
With a fixed and medium-sized wave number
$\kappa =128$, 
we start from the minimum oversampling parameter ($m=1$) and increase it gradually 
to observe the changes of the number of iterations.
For $Q^{(1)}_{m}$, we choose 
$h\sim \kappa^{-\frac{3}{2}}$, $H=\delta=\frac 12 H_{\mathrm{sub}}=\kappa^{-1} $,
while for $Q^{(2)}_{m}$, we select 
$h\sim \kappa^{-\frac{3}{2}}$, $H= \kappa^{-1}$,  $H_{\mathrm{sub}} = 4H$, $\delta=2H$.

The results are reported in Tables \ref{tab:table2-1} and \ref{tab:table2-2}. 
We see clearly that as $m$ increases, the number of iterations decreases rapidly and then becomes 
stabilized, for both $Q_{m}^{(1)}$ and $Q_{m}^{(2)}$, and $m=1$ or $2$ is already good enough. 
We have observed from many numerical experiments that it is reasonable to keep the oversampling parameter $m$ 
lower than the theoretical value (cf.\,Assumption \ref{asm:2}). 
This observation is also consistent with our theory, that is,  
the localized coarse space only needs to give a pre-asymptotic accuracy \cite{wu_pre-asymptotic_2014,zhu_preasymptotic_2013} rather than 
a pollution-free accuracy.
This is unlike the case when the LOD method is used for the discretization method, 
where $m$ must be taken to be large enough so that the localization error does not dominate \cite{peterseim_eliminating_2017,gallistl_stable_2015}, otherwise it would lead to a stagnation of the overall approximation error 
when the coarse mesh size $H$ decreases. 
\vskip-0.3truecm
\begin{table}[H]
\centering
\parbox{.45\linewidth}{
\centering
\caption{GMRES iterations of $Q^{(1)}_{m}$}
\vskip-0.3truecm
\begin{tabular}{cc}
\hline
$m$ & iter \\ \hline
 6   &   8   \\
 5   &   8   \\
 4   &   8   \\
 3   &   8   \\
 2   &   9   \\
 1   &   11   \\ \hline
\end{tabular}
\label{tab:table2-1}
}
\parbox{.45\linewidth}{
\centering
\caption{GMRES iterations of $Q^{(2)}_{m}$}
\vskip-0.3truecm
\begin{tabular}{cc}
\hline
$m$ & iter \\ \hline
 6   &   7   \\
 5   &   7   \\
 4   &   7   \\
 3   &   7   \\
 2   &   8   \\
 1   &   10   \\ \hline
\end{tabular}
\label{tab:table2-2}
}
\end{table}

\paragraph{{\bf Example 3}}
\label{para:exp3}
We now check how GMRES iterations depend on the subdomain sizes and the overlapping sizes of the subdomains.
We take $H_{0}=H_{0}(\kappa)$ to be the coarse mesh size such that $H_{0}\kappa \sim 1$.
Then for $Q_{m}^{(1)}$, we choose 
$h\sim \kappa^{-\frac{3}{2}}$, $H=\delta = \frac{1}{2}H_{\mathrm{sub}} $, \BowenRevise{$m=2$}, 
while for $Q_{m}^{(2)}$, we select 
$h\sim \kappa^{-\frac{3}{2}}$, $H=H_{0}(\kappa)$, $\delta = \frac{1}{2}H_{\mathrm{sub}} $, \BowenRevise{$m=2$}. \BowenRevise{To understand how the LOD coarse space is better than the classical $P_{1}$-FEM coarse space, we add the GMRES iteration numbers for $\tilde{Q}^{(i)}$, $i=1,2$, shown in brackets. $\tilde{Q}^{(i)}$ is given by $\tilde{Q}^{(1)} = Q_{H} + (I-Q_{H})^{T}\sum_{\ell} Q_{\ell} (I-Q_{H})$ and $\tilde{Q}^{(2)} = Q_{H} + (I-Q_{H})^{T}\sum_{\ell} \Pi_{h} \chi_{\ell} P_{\ell} (I-Q_{H})$, where $Q_{H}:V_{h}\to V_{H}$ is the $P_{1}$-FEM coarse solver. Here ``$-$'' in the table means the number is above $100$.
}

From Table \ref{tab:table3-1} we see 
the iteration numbers are roughly the same when subdomain sizes becomes smaller and hence the number of subdomains 
become bigger. This is highly consistent with the theoretical predictions. 
The results demonstrate good stability and scalability of the new preconditioners with the changes of the subdomain sizes. We see from Table \ref{tab:table3-2} that the performance of $Q_{m}^{(2)}$ 
improves when the overlapping size $\delta$ increases, but it does not help much 
as long as this size is decently large or proportional to $H$. 
\vskip-0.3truecm
{
\begin{table}[H]
\small
\centering
\parbox{.48\linewidth}{
\centering
\caption{GMRES iterations of $Q^{(1)}_{m}$ \BowenRevise{and $\tilde{Q}^{(1)}$ } }
\vskip-0.3truecm
\begin{tabular}{ccc}
\hline
$\kappa $ & $H_{\mathrm{sub}}=2H_{0}$ & $H_{\mathrm{sub}}=H_{0}$ \\ \hline
40                             &        $9\BowenRevise{(25)}$        &        $8\BowenRevise{(23)}$              \\
80                             &        $9\BowenRevise{(55)}$        &        $8\BowenRevise{(47)}$              \\
120                            &        $9\BowenRevise{(-)}$        &        $8\BowenRevise{(85)}$               \\
160                            &        $9\BowenRevise{(-)}$        &        $8\BowenRevise{(-)}$              \\ \hline
\end{tabular}
\label{tab:table3-1}
}
\hfill
\parbox{.5\linewidth}{
\centering
\caption{GMRES iterations of $Q^{(2)}_{m}$ \BowenRevise{and $\tilde{Q}^{(2)}$ } }
\vskip-0.3truecm
\begin{tabular}{ccc}
\hline
$\kappa$ & $\delta=2H_{0}$ & $\delta=4H_{0}$  \\ \hline
40                   &  $7\BowenRevise{(26)}$    &   $6\BowenRevise{(21)}$   \\
80                   &  $7\BowenRevise{(49)}$    &   $6\BowenRevise{(43)}$    \\
120                  &  $7\BowenRevise{(89)}$    &   $6\BowenRevise{(77)}$    \\
160                  &  $6\BowenRevise{(-)}$    &   $6\BowenRevise{(-)}$   \\ \hline
\end{tabular}
\label{tab:table3-2}
}

\end{table}
}


%

\paragraph{{\bf Example 4}}
\label{para:exp4}
In this example, we demonstrate how the iteration numbers of the preconditioner 
$Q_{m}^{(1)}$ when the overlap size decreases, especially when the overlap size 
is too smaller or even with the minimal overlap ($\delta=h$), 
namely, Assumption \ref{asm:1.1} is not satisfied. 
We choose the parameters to be
$h\sim \kappa_{\max}^{-3/2}$ ($\kappa_{\max}=120$), $H=H_{0}=\frac{1}{2}H_{\mathrm{sub}}$, $m=\log_{2}(\kappa)$, $H_{0}\kappa \sim 1$. We run the experiments for wave number $\kappa $ from $40$ and $120$, and keep $h$ fixed and small enough, 
i.e., $h\sim \kappa_{\max}^{-3/2}$. The results are shown in Table \ref{tab:table4-1}.
We can see from this table that the performance of the preconditioning effect of $Q_{m}^{(1)}$
deteriorates with the reduction of the overlap size as we would expect, but 
it is still quite stable for each $\kappa$ as long as the overlap size reduces to the range where it is not regarded as a decent overlap. 
\vskip-0.5truecm
\begin{table}[H]
    \centering
    \caption{GMRES iterations of $Q^{(1)}_{m}$}
    \vskip-0.3truecm
    \begin{tabular}{lcccc}
    \hline$\kappa$ & $\delta=H_{0}$ &  $\delta=4h$ & $\delta=2h$ & $\delta=h$  \\
    \hline 40 & 9 & 11 & 12 & 13  \\
    80 & 8 & 12 & 14 & 15  \\
    120 & 8 &  19 & 22 & 26  \\
    \hline
    \end{tabular}
    \label{tab:table4-1}
\end{table}

\section{Concluding remarks}
We have proposed and analyzed two two-level overlapping hybrid Schwarz preconditioners for solving 
the Helmholtz equation with high wave number in both two and three dimensions. Both preconditioners 
involve a global coarse solver (based on a localized orthogonal decomposition (LOD))  
and one local Helmholtz solver on each subdomain. 
It has remained a challenging topic how to design an effective coarse space for 
two-level Schwarz type preconditioners for the Helmholtz equation with high wave number, 
for which rigorous mathematical justifications of optimal convergence can be developed. 
No much progress has been made in this important direction during the past three decades. 
We are aware of a very recent work \cite{hu_novel_2024}, which appears to be the first one to
provide such a theoretical justification (but for two dimensions), 
based on a two-level preconditioner involving a global coarse subspace formed by eigenfunctions obtained on subdomains. 
Our current work seems to be the first one to develop the theoretical justification 
of the optimality of two-level overlapping 
Schwarz type preconditioners based on a more standard coarse solver (i.e., the LOD solver), 
in terms of all the key parameters, such as the fine mesh size, 
the coarse mesh size, the subdomain size and the wave numbers. Moreover, 
our coarse solvers possess a practically very important difference from those existing ones 
\cite{conen_coarse_2014,bootland_overlapping_2023,chupeng_wavenumber_2023,hu_novel_2024}, 
namely, the construction of our coarse solvers 
is completely independent of the local subproblems so the subdomains can vary and be replaced freely 
with no effect on the coarse solver and the optimality of the resulting preconditioners. 
Numerical experiments are presented to confirm the optimality and efficiency of the two proposed preconditioners.

Nevertheless, our convergence theory has also revealed 
an adverse effect of the two-level Schwarz type preconditioners 
for very large wave numbers, 
that is, the coarse space involved should be sufficiently large, as 
the coarse mesh size $H$ is required to be of the order $\kappa^{-1}$ theoretically. 
However, the coarse space would be still acceptable for small and medium-sized $\kappa$ 
in both two and three dimensions.
There are some potential strategies to reduce this adverse effect relatively for very large $\kappa$: 
(i) a multilevel Schwarz type preconditioner may be considered, instead of the two-level ones;
(ii) piecewise linear finite elements are used in the coarse solver and subdomain solvers  
when high order finite elements (that are preferred to reduce the pollution effect) are applied for the discretization. 
Furthermore, the preconditioners and theories developed in this work may be extended to 
some other important PDEs, such as the convection dominated 
diffusion model, the multiscale problems for which the LOD method was originally developed. 
These are interesting topics under progress. 


%


\bibliographystyle{siamplain}
\bibliography{refs}

\end{document}